\documentclass{amsart}

\usepackage{amsmath} 
\usepackage{amssymb}
\usepackage{mathrsfs}

\newtheorem{theorem}{Theorem}[section] 
\newtheorem{claim}[theorem]{Claim}

\theoremstyle{definition}
\newtheorem{definition}[theorem]{Definition}

\newtheorem{observation}[theorem]{Observation}

\newtheorem{question}[theorem]{Question}

\newtheorem{discussion}[theorem]{Discussion}

\newtheorem{hypothesis}[theorem]{Hypothesis}

\theoremstyle{remark}
\newtheorem{remark}[theorem]{Remark}

\newtheorem{conclusion}[theorem]{Conclusion}

\newcommand{\then}{{\underline{then}}}
\newcommand{\when}{{\underline{when}}}
\newcommand{\Then}{{\underline{Then}}}

\newcommand{\Iff}{{\underline{iff}}}
\newcommand{\mn}{{\medskip\noindent}}
\newcommand{\sn}{{\smallskip\noindent}}

\newcommand{\rest}{{\restriction}}
 
\newcommand{\Dom}{{\rm Dom}} 
\newcommand{\Depth}{{\rm Depth}} 
\newcommand{\alm}{{\rm alm}} 
\newcommand{\dual}{{\rm dual}} 
\newcommand{\ps}{{\rm ps}} 
\newcommand{\pry}{{\rm pry}} 
\newcommand{\id}{{\rm id}} 
\newcommand{\sal}{{\rm sal}} 
\newcommand{\Reg}{{\rm Reg}} 
\newcommand{\Rep}{{\rm Rep}} 
\newcommand{\Rang}{{\rm Rang}} 
\newcommand{\Ord}{{\rm Ord}} 
\newcommand{\rk}{{\rm rk}} 
\newcommand{\comp}{{\rm comp}} 
\newcommand{\hrtg}{{\rm hrtg}} 
\newcommand{\otp}{{\rm otp}} 
\newcommand{\ptf}{{\rm ptf}} 
\newcommand{\tcf}{{\rm tcf}} 
\newcommand{\wtcf}{{\rm wtcf}} 
\newcommand{\pcf}{{\rm pcf}} 
\newcommand{\AC}{{\rm AC}} 
\newcommand{\Ax}{{\rm Ax}} 
\newcommand{\RK}{{\rm RK}} 
\newcommand{\ZF}{{\rm ZF}}
\newcommand{\IND}{{\rm IND}}
\newcommand{\DC}{{\rm DC}} 
\newcommand{\Fil}{{\rm Fil}}

\newcommand{\wilog}{{\rm without loss of generality}}
\newcommand{\Wilog}{{\rm Without loss of generality}}

\newcommand{\cA}{{\mathscr A}}

\newcommand{\und}{\underline}

\newcommand{\gb}{{\mathfrak b}}
\newcommand{\ga}{{\mathfrak a}}
\newcommand{\gc}{{\mathfrak c}}

\newcommand{\bbD}{{\mathbb D}}

\newcommand{\gd}{{\mathfrak d\/}}

\newcommand{\cF}{{\mathscr F}}
\newcommand{\cG}{{\mathscr G}}

\newcommand{\cP}{{\mathscr P}}

\newcommand{\cU}{{\mathscr U}}

\newcommand{\cf}{{\rm cf}}

\newcount\skewfactor
\def\mathunderaccent#1#2 {\let\theaccent#1\skewfactor#2
\mathpalette\putaccentunder}
\def\putaccentunder#1#2{\oalign{$#1#2$\crcr\hidewidth
\vbox to.2ex{\hbox{$#1\skew\skewfactor\theaccent{}$}\vss}\hidewidth}}

 \newenvironment{PROOF}[2][\proofname.]
    {\begin{proof}[#1]\renewcommand{\qedsymbol}{\textsquare$_{\rm #2}$}}
    {\end{proof}}

\begin{document}

\title {Pseudo PCF}
\author {Saharon Shelah}
\address{Einstein Institute of Mathematics\\
Edmond J. Safra Campus, Givat Ram\\
The Hebrew University of Jerusalem\\
Jerusalem, 91904, Israel\\
 and \\
 Department of Mathematics\\
 Hill Center - Busch Campus \\ 
 Rutgers, The State University of New Jersey \\
 110 Frelinghuysen Road \\
 Piscataway, NJ 08854-8019 USA}
\email{shelah@math.huji.ac.il}
\urladdr{http://shelah.logic.at}
\thanks{The author thanks Alice Leonhardt for the beautiful typing.\\
Partially
supported by the United States-Israel Binational Science Foundation
(Grant No. 2006108).  Part of this work was done during the author's
visit to Mittag-Leffler Institute, Djursholm, Sweden.  He thanks the
Institute for hospitality and support, Publication 955.}  



\date{November 6, 2011}

\begin{abstract}
We continue our investigation on pcf with weak form of choice.
Characteristically we assume DC + $\cP(Y)$ when looking and
$\prod\limits_{s \in Y} \delta_s$.  We get more parallel of theorems
on pcf.
\end{abstract}

\maketitle
\numberwithin{equation}{section}
\setcounter{section}{-1}

\centerline {Anotated Content}
\bigskip

\noindent
\S0 \quad Introduction, pg.2
\bigskip

\noindent
\S1 \quad On pseudo true cofinality, pg.4
\mn
\begin{enumerate}
\item[${{}}$]  [We continue \cite[\S5]{Sh:938} to try to 
generalize the pcf theory for $\aleph_1$-complete filters
$D$ on $Y$ assuming only DC + AC$_{\cP(Y)}$.  So is similar to
\cite[ChXII]{Sh:b}.  We suggest to replace cofinality by pseudo
cofinality.  In particular we get the existence of a sequence of
generators, get a bound to Reg $\cap$ pp$(\mu) \backslash \mu_0$ using
the size of Reg $\cap \mu \backslash \mu_0$ using a no-hole claim and
existence of lub (unlike \cite{Sh:835}).   
\end{enumerate}
\bigskip

\noindent
\S2 \quad Composition and generating sequences for pseudo pcf, pg.14
\mn
\begin{enumerate}
\item[${{}}$]  [We deal with pseudo true cofinality of
$\prod\limits_{i \in Z} \, \prod\limits_{j \in Y_i} \lambda_{i,j}$ and
below also in the degenerated case each $\langle \lambda_{i,j}:j \in
Y_i\rangle$ is constant.  We then use it to clarify the state of
generating sequences; see \ref{e1}, \ref{e2}, \ref{e23},
\ref{e3}, \ref{e31}, \ref{e33}.
\end{enumerate}
\bigskip

\noindent
\S3 \quad Measuring Reduced products, pg.23
\bigskip

\noindent
\S (3A) \quad On ps-$\bold T_D(g)$
\mn
\begin{enumerate}
\item[${{}}$]  [We get that several measures of 
${}^\kappa \mu/D$ are essentially equal.] 
\end{enumerate}
\bigskip

\noindent
\S (3B) \quad Depth of reduced power of ordinals, pg.26
\mn
\begin{enumerate}
\item[${{}}$]  [Using the independence property for a sequence of
filters we can bound the relevant depth.  This generalizes
\cite{Sh:460} or really \cite[\S3]{Sh:513}.]
\end{enumerate}
\bigskip

\noindent
\S (3C) \quad Bounds on the Depth
\mn
\begin{enumerate}
\item[${{}}$]  [We start by basic properties by No-Hole Claim (\ref{r22}(1))
and dependence of $\bar\alpha/D$ only (\ref{c4}).  We give a bound for
$\lambda^{+\alpha(1)}/D$ (in Theorem \ref{c7}, \ref{c13}).]
\end{enumerate}
\bigskip

\noindent
\S (3D) \quad Concluding Remarks
\mn
\begin{enumerate}
\item[${{}}$]  [Comments to \cite{Sh:938}.]
\end{enumerate}
\bigskip

\noindent
\S4 \quad On RGCH with Little Choice
%\mn
%\begin{enumerate}
%\item[${{}}$]  [ ]
%\end{enumerate}
\newpage

\section {Introduction}

In the first section we deal with generalizing the pcf theory in the
direction started in \cite[\S5]{Sh:938} trying to understand pseudo
true cofinality of small products of regular cardinals.  
The difference with earlier works is that here we assume AC$_{\cU}$ for any
set $\cU$ of power $\le |\cP(\cP(Y))|$ or actually working harder,
just $\le |\cP(Y)|$ when analyzing 
$\prod\limits_{t \in Y} \alpha_t$, whereas in \cite{Sh:497} we 
assumed AC$_{\sup\{\alpha_t:t \in Y\}}$ and in \cite{Sh:835} we have
(in addition to AC$_{\cP(\cP(Y))}$) assumptions like
``${}^\omega(\sup\{\alpha_t:t \in Y\})$ is well ordered".  In 
\cite[\S1-\S4]{Sh:938} we assume only AC$_{< \mu} + DC$ and consider
$\aleph_1$-complete filters on $\mu$ but in the characteristic case
$\mu$ is a limit of measurable cardinals.

Note that generally in this work, though we try occasionally not to
use DC, it will not be a real loss to assume it all the time.  More
specifically, we prove the existence of a minimal $\aleph_1$-complete
filter $D$ on $Y$ such that $\lambda = \text{ ps-tcf}(\Pi
\bar\alpha,<_D)$ assuming AC$_{\cP(Y)}$ and (of course DC and)
$\alpha_t$ of large enough cofinality.  We then prove the existence of
$X \subseteq Y$ such that $J^{\aleph_1\text{-comp}}_{\le
\lambda}[\bar\alpha] = J^{\aleph_1\text{-comp}}_{<
\lambda}[\bar\alpha] +X$, see \ref{r18} and even (in \ref{r19}) the
parallel of existence $a$ a $<_{D_1}$-lub for an $<_D$-increasing
sequence $\langle \cF_\alpha:\alpha < \lambda\rangle$, 
generalize the no-hole claim \ref{r22}, and give a bound on pp for
non-fix points (\ref{r21}).

In \S2 we further investigate true cofinality.  In Claim \ref{e2},
assuming AC$_\lambda$ and $D$ an $\aleph_1$-complete filter on $Y$, we start
from ps-tcf$(\Pi \bar \alpha,<_D)$, dividing by eq$(\bar\alpha) =
\{(s,t):\alpha_s = \alpha_t\}$.  We also prove the composition
Theorem \ref{e3}: when ps-tcf$(\prod\limits_{i} \text{
ps-tcf}(\prod\limits_{j} \lambda_{i,j},<_{D_i}),<_E)$ is
ps-tcf$(\prod\limits_{(i,j)} \lambda_{i,j},<_D)$.

We then prove the pcf closure conclusion: giving a sufficient
condition for the operation ps-pcf$_{\aleph_1\text{-comp}}$ to be idempotent.
Lastly, we revisit the generating sequence.

In \S(3A) we measure $\prod\limits_{t \in Y} g(t)$ for $g \in
{}^Y(\text{Ord} \backslash \{0\})$ in three ways and show they are
almost equal in \ref{r24}.  The price is
that we replace (true) cofinality by pseudo (true) cofinality, which
is inevitable.

In \S (3B) we 
prove a relative of \cite[\S3]{Sh:513}; again dealing with depth
(instead of rank as in \cite{Sh:938}) 
adding some information even under ZFC.  Assuming that 
the sequence $\langle D_n:n < \omega\rangle$ of filters
has the independence property (IND), see Definition 
\ref{k4}, with $D_n$ a filter
on $Y_n$ we can bound the depth of ${}^{(Y_n)}\zeta$ by $\zeta$, for every
$\zeta$ for  many $n$'s, see \ref{k6}.  Of
course, we can generalize this to $\langle D_s:s \in S\rangle$.  
This is incomparable with the results of
\cite[\S4]{Sh:938}; also we add some cases to \cite[\S4]{Sh:938}.

Note that the assumptions like IND$(\bar D)$ are complimentary to ones used
in \cite{Sh:835} to get considerable information.
Our original hope was to arrive to a dichotomy.  
The first possibility will say that one of
the versions of an axiom suggested in \cite{Sh:835} holds, which means
``for some suitable algebra", there is no independent
$\omega$-sequence; in this case \cite{Sh:835} tells us much.  The
second possibility will be a case of IND, and then we try to show that
there is a rank system in the sense of \cite{Sh:938}.  But presently
for this we need too much choice.  The dichotomy we succeed to prove is
with small o-Depth in
one side, the results of \cite{Sh:835} on the other side.  It would be
better to have ps-o-Depth in the first side.
We try to sort out the ``almost equal" in \ref{r32} - \ref{r36}. 
 
\begin{question}
\label{r15} 
[DC + AC$_{\cP(Y)}$]

Assume
\medskip

\noindent
\begin{enumerate}
\item[$(a)$]   $\bar\alpha \in {}^Y$ {\rm Ord}
\smallskip

\noindent
\item[$(b)$]   cf$(\alpha_t) \ge \hrtg({\cP}(Y)$ for every $t \in Y$
\smallskip

\noindent
\item[$(c)$]  $\lambda_t \in 
\text{ pcf}_{\aleph_1-\comp}(\bar\alpha)$ for $t \in Z$, in fact,
$\lambda_t = \text{ ps-tcf}(\Pi \bar \alpha,<_{D_t}),D_t$ is a
$\aleph_1$-complete filter on $Y$
\smallskip

\noindent
\item[$(d)$]   $\lambda =$ {\rm ps-tcf}$_{\aleph_1-\comp}
(\langle \lambda_t:t \in Z\rangle$
\smallskip

\noindent
\item[$(e)$]   (a possible help) $X_t \in D_t,\langle X_t:t \in Y\rangle$ are
pairwise disjoint.
\end{enumerate}
\medskip

\noindent
$(A) \quad$ Now does $\lambda \in 
\text{ ps-pcf}_{\aleph_1\text{-comp}}(\bar \alpha)$?  Can we say
something on $D_\lambda$ from 

\hskip25pt \cite[5.9]{Sh:938} improved in \ref{r16}
\medskip

\noindent
$(B)\quad$ At least when trying to generalize the RGCH, see
\cite{Sh:460} and \cite{Sh:829}.
\end{question}
\newpage

\section {On pseudo true cofinality}

We continue \cite[\S5]{Sh:938}. 

Below we give an improvement of \cite[5.19]{Sh:938}, omitting {\rm
DC} from the assumptions but first we observe
\begin{claim}
\label{r15}
Assume {\rm AC}$_Z$.

\noindent
1) We have $\theta \ge \hrtg(Z)$ \when \,
   $\bar\alpha = \langle \alpha_t:t \in Y\rangle$ and $\theta \in$ 
{\rm ps-pcf}$(\Pi \bar \alpha)$ and $t \in Y \Rightarrow$ {\rm cf}$(\alpha_t) 
\ge \hrtg(Z)$.

\noindent
2) We have $\cf(\rk_D(\bar\alpha))) \ge \hrtg(Z)$ \when
   \, $\bar\alpha = \langle \alpha_t:t \in Y\rangle,t \in Y
   \Rightarrow \cf(\alpha_t) \ge \hrtg(Z)$.
\end{claim}

\begin{PROOF}{\ref{r15}}
Clearly $\theta$ is a regular cardinal.

\noindent
1) If we have AC$_\alpha$ for every $\alpha < \hrtg(Z)$ then we can
   use \cite[5.7(4)]{Sh:938} but we do not assume this.  
In general let $D$ be an
$\aleph_1$-complete filter on $Y$ such that 
$\theta =$ {\rm ps-tcf}$(\Pi \bar \alpha,<_D)$, exists as we are
 assuming $\theta \in$ {\rm ps-pcf}$(\Pi \bar \alpha)$.  Let
   $\bar{\cF} = \langle \cF_\alpha:\alpha < \theta \rangle$ witness
$\theta =$ {\rm ps-tcf}$(\Pi \bar \alpha,<_D)$, i.e. as in
\cite[5.6(2)]{Sh:938} note as $t \in Y \Rightarrow \alpha_t >
   0$, we are assuming $\cF_\alpha \subseteq \Pi \bar\alpha$; also we
   can assume $\cF_\alpha \ne \emptyset$ for every $\alpha < \theta$.

Toward contradiction assume $\theta := \cf(\theta) < \hrtg(Z)$.
As $\theta < \hrtg(Z)$, there is a function $h$ from $Z$
onto $\theta$, so the sequence $\langle \cF_{h(z)}:z \in
Z\rangle$ is well defined.  As we are assuming AC$_Z$, there is a
sequence $\langle f_z:z \in Z\rangle$ such that $f_z \in
\cF_{h(z)}$ for $z \in Z$.  Now define $g \in
{}^Y(\text{Ord})$ by $g(s) = \cup\{f_z(s):z \in Z\}$; clearly $g$
exists and $g \le \bar\alpha$.  But for each $s \in Y$, the set
$\{f_z(s):z \in Z\}$ is a subset of $\alpha_s$ of cardinality $<
\hrtg(Z)$ hence $< \text{ cf}(\alpha_s)$ hence $g(s) < \alpha_s$.  Together $g
\in \Pi\bar\alpha$ is a $<_D$-upper bound of
$\cup\{\cF_\varepsilon:\varepsilon < \theta\}$, 
contradiction to the choice of $\bar{\cF}$.

\noindent
2) Otherwise let $\theta < \hrtg(Z),\langle
   \alpha_\varepsilon:\varepsilon < \theta \rangle$ be increasing with
   limit rk$_D(\bar\alpha)$ and again let $g$ be a function from $Z$
   onto $\theta$.  As AC$_Z$ holds, we can find $\langle f_z:z \in Z 
\rangle$ such that for every $z \in Z$ we have 
rk$_D(f_z) \ge \alpha_{h(z)}$ and $f_z <_D \bar\alpha$ and 
\wilog $f_z \in \Pi\bar\alpha$.  Let $f
   \in \Pi\bar\alpha$ be $f(t) = \sup\{f_{h(z)}(t):z \in Z\}$ 
so rk$_D(f) \ge \sup\{\alpha_z:z \in Z\} 
= \rk_D(\bar\alpha) > \rk_D(f)$, contradiction.
\end{PROOF}

\begin{theorem}
\label{r16} 
\underline{The Canonical Filter Theorem}
Assume {\rm AC}$_{\cP(Y)}$.

Assume $\bar \alpha = \langle \alpha_t:t \in Y\rangle \in
{}^Y\text{\rm Ord}$ and $t \in Y \Rightarrow \cf(\alpha_t) 
\ge \hrtg({\cP}(Y))$ and
$\partial \in$ {\rm ps-pcf}$_{\aleph_1-\comp}(\bar \alpha)$
hence is a regular cardinal.  \Then \, there is $D =
D^{\bar\alpha}_\partial$, an $\aleph_1$-complete filter on $Y$ such that 
$\partial =$ {\rm ps-tcf}$(\Pi \bar \alpha/D)$ and $D \subseteq D'$
for any other such $D' \in \text{\rm Fil}^1_{\aleph_1}(D)$.
\end{theorem}

\begin{remark}
\label{c17d}
1) By \cite[5.9]{Sh:938} there are some such $\partial$ if DC holds.

\noindent
2) We work more to use just AC$_{\cP(Y)}$ and not more.

\noindent
3) If $\kappa >\aleph_0$ we can replace ``$\aleph_1$-complete" by 
``$\kappa$-complete". 

\noindent
4) If we waive ``$\partial$ regular" so just $\partial$, an ordinal,
   is a pseudo true cofinality of $(\Pi \bar\alpha,<_D)$ for $D, \in
   \bbD \subseteq \Fil^1_{\aleph_1}(Y)$, exemplified by
   $\bar{\cF}^D,\bbD \ne \emptyset$ the proof gives some
   $\partial',\cf(\gamma') = f(\gamma)$ and $\bar{\cF}$ witnessing
   $(\Pi \bar\alpha,<_{D_*})$ has pseudo true cofinality where $D_* =
   \cap\{D:D \in \bbD\}$.
\end{remark}

\begin{PROOF}{\ref{r16}}   
Note that by \ref{r15}
\mn
\begin{enumerate}
\item[$\boxplus_1$]  $\oplus \quad \partial \ge \hrtg(\cP(Y))$.
\end{enumerate}
\mn
Let 
\medskip

\noindent
\begin{enumerate}
\item[$\boxplus_2$]   $(a) \quad \bbD = \{D:D$ is an 
$\aleph_1$-complete filter on $Y$ such that $(\Pi\bar\alpha/D)$ has

\hskip40pt   pseudo true cofinality $\partial\}$,
\smallskip

\noindent
\item[${{}}$]   $(b) \quad D_* = \cap\{D:D \in \bbD\}$.
\end{enumerate}
\mn
Now obviously
\medskip

\noindent
\begin{enumerate}
\item[$\boxplus_3$]   $(a) \quad \bbD$ is non-empty
\sn
\item[${{}}$]  $(b) \quad D_*$ is an $\aleph_1$-complete filter on $Y$.
\end{enumerate}
\mn
For $A \subseteq Y$ let $\bbD_A = \{D \in \bbD:A \notin D\}$ and
let ${\cP}_* = \{A \subseteq Y:\bbD_A \ne \emptyset\}$, equivalently
$\cP_* = \{A \subseteq Y:A \ne \emptyset$ mod $D\}$.  As
AC$_{{\cP}(Y)}$ holds we can find $\langle D_A:A \in {\cP}_*\rangle$
such that $D_A \in \bbD_A$ for $A \in {\cP}_*$.  Let $\bbD_* =
\{D_A:A \in {\cP}_*\}$, clearly
\medskip

\noindent
\begin{enumerate}
\item[$\boxplus_4$]  $(a) \quad D_* = \cap\{D:D \in \bbD_*\}$
\sn
\item[${{}}$]  $(b) \quad \bbD_* \subseteq \bbD$ is non-empty.
\end{enumerate}
\medskip

\noindent
As AC$_{\cP_*}$ holds clearly
\medskip

\noindent
\begin{enumerate}
\item[$(*)_1$]   we can 
choose $\langle \bar{\cF}^A:A \in \cP_*\rangle$ such that 
$\bar{\cF}^A$ exemplifies $D_A \in \bbD$ as in 
\cite[5.17,(1),(2)]{Sh:938}, so in particular is $\aleph_0$-continuous
and \wilog \, $\cF^A_\alpha \ne \emptyset,\cF^A_\alpha \subseteq
\Pi\bar\alpha$ for every $\alpha < \partial$.
\end{enumerate}

For each $\beta < \partial$ let
\mn
\begin{enumerate}
\item[$(*)_2$]   $\bold F^1_\beta = \{\bar f = \langle f_A:A \in
\cP_* \rangle:\bar f$ satisfies $A \in \cP_* \Rightarrow f_A \in
\cF^A_\beta\}$
\sn
\item[$(*)_3$]   for $\bar f \in \bold F^1_\beta$ let sup$\{f_A:A \in
\cP_*\}$ be the function $f \in {}^Y\text{Ord}$ defined by $f(y) =
\sup\{f_A(y):A \in \cP_*\}$
\sn
\item[$(*)_4$]   $\cF^1_\beta = \{\sup\{f_A:A \in \cP_*\}:\bar f =
\langle f_A:A \in \cP_*\rangle$ belongs to $\bold F^1_\beta\}$.
\end{enumerate}
\mn
Now
\mn
\begin{enumerate}
\item[$(*)_5$]  $(a) \quad \langle \cF^1_\beta:\beta <
\partial\rangle$ is well defined, i.e. exist
\sn
\item[${{}}$]  $(b) \quad {\cF}^1_\beta \subseteq \Pi \bar\alpha$.
\end{enumerate}
\mn
[Why?  Clause (a) holds by the definitions, clause (b) holds as 
$t \in Y \Rightarrow \text{ cf}(\alpha_t) \ge \hrtg(\cP(Y))$.]
\mn
\begin{enumerate}
\item[$(*)_6$]  $\cF^1_\beta \ne \emptyset$ for $\beta < \partial$.
\end{enumerate}
\mn
[Why?  As for $\beta < \lambda$, the sequence $\langle
\bar{\cF}^A_\beta:A \in \cP_*\rangle$ is well defined
(as $\langle \bar{\cF}^A:A \in \cP_*\rangle$ is) and $A \in \cP_*
\Rightarrow \cF^A_\alpha \ne \emptyset$, so we can use AC$_{\cP(Y)}$ to
deduce $\cF^1_\beta \ne \emptyset$.]

Define
\mn
\begin{enumerate}
\item[$(*)_7$]  $(a) \quad$ for $f \in \Pi \bar\alpha$ and $A \in \cP_*$ let

\hskip25pt $\beta_A(f) = \min\{\beta < \lambda:f < 
g \text{ mod } D_A$ for every $g \in \cF^A_\beta\}$
\sn
\item[${{}}$]  $(b) \quad$ for $f \in \Pi \bar\alpha$ 
let $\beta(f) = \sup\{\beta_A(f):A \in \cP_*\}$.
\end{enumerate}
\mn
Now
\mn
\begin{enumerate}
\item[$(*)_8$]  for $A \in \cP_*$ and $f \in 
\Pi\bar\alpha$, the ordinal $\beta_A(f) < \partial$ is well defined.
\end{enumerate}
\mn
[Why?  As $\langle \cF^A_\gamma:\gamma < \partial\rangle$ is cofinal in
$(\Pi,\bar\alpha,<_{D_A})$.] 
\mn
\begin{enumerate}
\item[$(*)_9$]  $(a) \quad$ for $f \in \Pi \bar\alpha$ the
ordinal $\beta(f)$ is well defined and $< \partial$
\sn
\item[${{}}$]  $(b) \quad$ if $f \le g$ are from $\Pi \bar\alpha$ then
$\beta(f) \le \beta(g)$.
\end{enumerate}
\mn
[Why?  For clause (a), first, $\beta(f)$ is well defined and 
$\le \partial$ by $(*)_8$ and the definition of $\beta(f)$ in
$(*)_7(b)$.  
Second, recalling that $\partial$ is regular $\ge
\hrtg(\cP(Y)) \ge \hrtg(\cP_*)$ clearly $\beta(f) < \partial$.  Clause
(b) is obvious.]

Now
\mn
\begin{enumerate}
\item[$(*)_{10}$]  $(a) \quad$ if $A \in \cP_*,\gamma < \partial$ and
$f \in \cF^A_\gamma$ then $\beta_A(f) > \gamma$
\sn
\item[${{}}$]  $(b) \quad$ if $\gamma < \partial$ and $f \in
\cF^1_\gamma$ then $\beta_A(f) > \gamma$.
\end{enumerate}
\mn
[Why?  Clause (a) holds because $\beta < \gamma \wedge g \in
\cF^A_\gamma \Rightarrow g < f$ mod $D_A$ and $\beta = \gamma
\Rightarrow f \in \cF^A_\gamma \wedge f \nleq f$ mod $D_A$.  Clause
(b) holds because for some $\langle f_B:B \in \cP_*\rangle \in
\Pi\{F^B_\gamma:B \in \cP_*\}$ we have $f = \sup\{f_B:B \in \cP_*\}$
hence $B \in \cP_* \Rightarrow f_B \le f$ hence in particular $f_A \le
f$, recalling $\beta(f_A) > \gamma$ by clause (a) it follows that
$\beta(f) > \gamma$.]
\mn
\begin{enumerate}
\item[$(*)_{11}$]  $(a) \quad$ for $\xi < \partial$ let
$\gamma_\xi = \min\{\beta(f):f \in \cF^1_\xi\}$
\sn
\item[${{}}$]  $(b) \quad$ for $\xi < \partial$ let $\cF^2_\xi = \{f
\in \cF^1_\xi:\beta(f) = \gamma_\xi\}$
\sn
\item[$(*)_{12}$]   $(a) \quad \langle (\gamma_\xi,\cF^2_\xi):\xi < \partial
\rangle$ is well defined, i.e. exists
\sn
\item[${{}}$]  $(b) \quad$ if $\xi < \partial$ then $\xi < \gamma_\xi <
\partial$.
\end{enumerate}
\mn
[Why?  $\gamma_\xi$ is the minimum of a set of ordinals which is
non-empty by $(*)_6$ and $\subseteq \partial$, by $(*)_8$, and all
members are $> \gamma$ by $(*)_{10}(a)$.]
\mn
\begin{enumerate}
\item[$(*)_{13}$]  for $\xi < \partial$ we have $\cF^2_\xi \subseteq
\Pi \bar\alpha$ and $\cF^2_\xi \ne \emptyset$.
\end{enumerate}
\mn
[Why?  By $(*)_{11}$ as $\cF^1_\xi \ne \emptyset$ and $\cF^1_\xi
\subseteq \Pi\bar\alpha$.]
\mn
\begin{enumerate}
\item[$(*)_{14}$]  we try to define $\beta_\varepsilon < \partial$ by
induction on the ordinal $\varepsilon < \partial$

\underline{$\varepsilon = 0$}:  $\beta_\varepsilon = 0$

\underline{$\varepsilon$ limit}: $\beta_\varepsilon =
\cup\{\beta_\zeta:\zeta < \varepsilon\}$

\underline{$\varepsilon = \zeta +1$}:  $\beta_\varepsilon =
\gamma_{\beta_\zeta}$
\sn
\item[$(*)_{15}$]  $(a) \quad$ if $\varepsilon < \partial$ then
$\beta_\varepsilon < \partial$ is well defined $\ge \varepsilon$
\sn
\item[${{}}$]  $(b) \quad$ if $\zeta < \varepsilon$ is well defined
then $\beta_\zeta < \beta_\varepsilon$.
\end{enumerate}
\mn
[Why?  Clause (a) holds as $\partial$ is a regular cardinal
so the case $\varepsilon$
limit is O.K., the $\varepsilon = \zeta +1$ holds by $(*)_{12}$.  As for
clause (b) recall $\beta_A(f) > \xi$ for $f \in \cF^1_\xi$ by $(*)_{10}(b)$.]
\mn
\begin{enumerate}
\item[$(*)_{16}$]  if $f \in \Pi\bar\alpha$, then for some $g \in
\cup\{\cF^2_{\beta_\varepsilon}:\varepsilon < \partial\}$ we have $f <
g$ mod $D$.
\end{enumerate}
\mn
[Why?  Recall that $\beta_A(f)$ for $A \in \cP_*$ and $\beta(f)$ are
well defined ordinals $<\partial$ and let $\zeta < \partial$ be such
that $\beta(f) < \beta_\zeta$, exists as $\beta_\varepsilon \ge
\varepsilon$.  As $\bar{\cF}^A$ is $<_{D_A}$-increasing for $A \in \cP_*$
clearly $A \in \cP_* \wedge g \in \cF^A_{\beta_\zeta} \Rightarrow f <
g$ mod $D_A$.  So by the definition of $\cF^1_{\beta_\zeta}$ we have
$A \in \cP_* \wedge g \in \cF^1_{\beta_\zeta} \Rightarrow f < g$ mod
$D_A$ hence $g \in \cF^1_{\beta_\zeta} \Rightarrow f < g$ mod $D_*$.  As
$\cF^2_{\beta_\zeta} \subseteq \cF^1_{\beta_\zeta}$ we are done.]
\mn
\begin{enumerate}
\item[$(*)_{17}$]  if $\zeta < \xi < \partial$ and 
$f \in \cF^2_\zeta$ and $g \in \cF_\xi$ then $f < g$ mod $D_*$.
\end{enumerate}
\mn
[Why?  As in the proof of $(*)_{16}$ but now $\beta(f) = \gamma_\zeta$.]

Together by $(*)_{13} + (*)_{16} + (*)_{17}$ the sequence
$\langle \cF^2_{\beta_\varepsilon}:\varepsilon <
\partial\rangle$ is as required.
\end{PROOF}

\noindent
A central definition here is
\begin{definition}
\label{r17g}
1) For $\bar\alpha \in {}^Y\text{Ord}$
let $J^{\aleph_1\text{\rm -comp}}_{< \lambda}[\bar \alpha] = \{X \subseteq
Y$: ps-pcf$_{\aleph_1-\comp}(\bar \alpha \restriction X) 
\subseteq \lambda\}$.  So for $X \subseteq Y,X \notin
J^{\aleph_1-\comp}_{< \lambda}[\bar\alpha]$ iff there is an
$\aleph_1$-complete filter $D$ on $Y$ such that $X \ne \emptyset$ mod
$D$ and ps-tcf$(\Pi \bar\alpha,<_D)$ is well defined $\ge \lambda$
\Iff \, there is an $\aleph_1$-complete filter $D$ on $Y$ such that
ps-tcf$(\Pi \bar\alpha,<_D)$ is well defined $\ge \lambda$ and $X \in D$.

\noindent
2)  $J^{\aleph_1-\comp}_{\le \lambda}$ is 
$J^{\aleph_1-\comp}_{< \lambda^+}$ and we can use a set $\ga$ of
ordinals instead of $\bar\alpha$.
\end{definition}

\begin{claim}
\label{r18} 
\underline{The Generator Existence Claim} 

\noindent
Let $\bar\alpha \in {}^Y(\text{\rm Ord} \backslash \{0\})$.

\noindent
1)  $J^{< \aleph_1-\comp}_{< \lambda}(\bar \alpha)$ is an
$\aleph_1$-complete ideal on $Y$ for any cardinal $\lambda$ except
that it may be $\cP(Y)$.

\noindent
2) [{\rm AC}$_{{\cP}(Y)}$]  Assume $t \in Y \Rightarrow \text{\rm
cf}(\alpha_t) \ge \hrtg(\cP(Y))$.  If $\lambda \in$ 
{\rm ps-pcf}$_{\aleph_1-\comp}(\bar \alpha)$ 
\then \, for some $X \subseteq Y$ we have
\medskip

\noindent
\begin{enumerate}
\item[$(A)$]   $J^{\aleph_1-\comp}_{< \lambda^+}[\bar \alpha]
= J^{\aleph_1-\comp}_{< \lambda}[\bar \alpha] + X$
\smallskip

\noindent
\item[$(B)$]  $\lambda =$ {\rm ps-tcf}$(\Pi\bar \alpha,
<_{J^{\aleph_1-\comp}_{=\lambda}[\bar \alpha]})$ where
$J^{\aleph_1-\comp}_{=\lambda}[\bar \alpha] := 
J^{\aleph_1-\comp}_{< \lambda}[\bar \alpha] + (Y \backslash X)$
\smallskip

\noindent
\item[$(C)$]  $\lambda \notin$ {\rm ps-pcf}$_{\aleph_1-\comp}
(\bar \alpha \restriction (Y \backslash X))$.
\end{enumerate}
\end{claim}

\begin{remark}
\label{r18d}
Recall that if AC$_{\cP(Y)}$ \then \, \wilog \, 
AC$_{\aleph_0}$ holds. Why?  Otherwise by
AC$_{\cP(Y)}$ we have $Y$ is well ordered and AC$_Y$ hence $|Y| = n$
for some $n < \omega$ and in this case our claims are obvious,
e.g. \ref{r18}(2), \ref{r19}.
\end{remark}

\begin{PROOF}{\ref{r18}}   1)  If not then we can find a sequence
$\langle A_n:n < \omega\rangle$ of members of
$J^{\aleph_1-\comp}_{< \lambda}[\bar\alpha]$ such that their
union $A := \cup\{A_n:n < \omega\}$ does not belong to it.  As $A
\notin J^{\aleph_1-\comp}_{< \lambda}[\bar\alpha]$, by the
definition there is an $\aleph_1$-complete filter $D$ on $Y$ such that
$A \ne \emptyset$ mod $D$ and ps-tcf$(\Pi \bar\alpha,<_D)$ is well
defined, so let it be $\mu = \cf(\mu) \ge \lambda$ and let
$\langle \cF_\alpha:\alpha < \lambda\rangle$ exemplify it.  

As $D$ is $\aleph_1$-complete and $A = \cup\{A_n:n < \omega\} \ne
\emptyset$ mod $D$ necessarily for some $n,A_n \ne \emptyset$ mod $D$
but then $D$ witness $A_n \notin 
J^{\aleph_1-\comp}_{< \lambda}[\bar\alpha]$, contradiction.

\noindent
2) Recall $\lambda$ is a regular cardinal by
\cite[5.8(0)]{Sh:938} and $\lambda \ge \hrtg({\cP}(Y))$ by \ref{r15}. 

Let $D = D^{\bar \alpha}_\lambda$ be as in \cite[5.19]{Sh:938}
when DC holds, and as in \ref{r16} in general,
i.e. $\Pi \bar\alpha/D$ has pseudo true cofinality $\lambda$ and $D$
contains any other such $\aleph_1$-complete 
filter on $Y$.  Now if $X \in D^+$ then $\lambda =$ 
{\rm ps-tcf}$_{\aleph_1-\comp}(\bar \alpha \restriction X,<_{(D+X) \cap
{\cP}(X)})$ hence $X \notin J^{\aleph_1-\comp}_{< \lambda} [\bar \alpha]$, so
\medskip

\noindent
\begin{enumerate}
\item[$(*)_1$]   $X \in J^{\aleph_1-\comp}_{< \lambda}[\bar \alpha] 
\Rightarrow X = \emptyset$ mod $D$.
\end{enumerate}
\mn
A major point is
\medskip

\noindent
\begin{enumerate}
\item[$(*)_2$]    some $X \in D$ belongs to
$J^{\aleph_1-\comp}_{< \lambda^+}[\bar \alpha]$.
\end{enumerate}
\mn
Why $(*)_2$?  The proof will take awhile; assume that not,  
we have AC$_{{\cP}(Y)}$ hence AC$_D$, so we can
find $\langle (\bar{\cF}^X,D_X,\lambda_X):X \in D\rangle$ such that:
\medskip

\noindent
\begin{enumerate}
\item[$(a)$]   $\lambda_X$ is a regular cardinal $\ge \lambda^+$,
i.e. $> \lambda$
\smallskip

\noindent
\item[$(b)$]  $D_X$ is an $\aleph_1$-complete filter on $Y$ such
that $X \in D_X$ and 

$\lambda_X =$ {\rm ps-tcf}$(\Pi \bar \alpha,<_{D_X})$
\smallskip

\noindent
\item[$(c)$]  $\bar{\cF}^X = \langle {\cF}^X_\alpha:
\alpha < \lambda_X\rangle$ exemplifies that $\lambda_X =$ {\rm ps-tcf}
$(\Pi \bar \alpha,<_{D_X})$
\smallskip

\noindent
\item[$(d)$]  moreover $\bar{\cF}^X$ is as in \cite[5.17(2)]{Sh:938}.
\end{enumerate}
\mn
Let
\medskip

\noindent
\begin{enumerate}
\item[$(e)$]  $D^*_1 = \{A \subseteq Y$: for some $X_1 \in D$ we
have $X \in D \wedge X \subseteq X_1 \Rightarrow A \in D_X\}$.
\end{enumerate}
\mn
Clearly
\medskip

\noindent
\begin{enumerate}
\item[$(f)$]  $D^*_1$ is an $\aleph_1$-complete filter on $Y$
extending $D$.
\end{enumerate}
\medskip

\noindent
[Why?  First, clearly $D^*_1 \subseteq \cP(Y)$ and $\emptyset \notin D^*_1$ as
 $X \in D \Rightarrow \emptyset \notin D_X$.  Second,
 if $A \in D$ then $X \in D \wedge X \subseteq A
\Rightarrow A \in D_X$ by clause (b) 
hence choosing $X_1 = A$ the demand for $``A \in
D^*_1"$ holds so indeed $D \subseteq D^*_1$.  Third, assume 
$\bar A = \langle A_n:n < \omega\rangle$ and $``A_n \in
D^*_1"$ for $n < \omega$, \underline{then} for each $A_n$ there is a witness 
$X_n \in D$, so by AC$_{\aleph_0}$, recalling \ref{r18d}, 
there is an $\omega$-sequence $\langle X_n:n
<\omega\rangle$ with $X_n$ witnessing $A_n \in D^*_1$.  Then 
$X = \cap\{X_n:n < \omega\}$ belongs to $D$ and witness
that $A := \cap\{A_n:n < \omega\} \in D^*_1$ because $D$ is
$\aleph_1$-complete.  Fourth, if $A \subseteq B \subseteq Y$ and $A
 \in D^*_1$, then some $X_1$ witness $A \in D^*_1$, i.e. $X \in D
 \wedge X \subseteq X_1 \Rightarrow A \in D_X$; but then $X_1$ witness
 also $B \in D^*_1$.]
\medskip

\noindent
\begin{enumerate}
\item[$(g)$]   assume $\langle \cF_\alpha:\alpha < \lambda\rangle$ is
$<_D$-increasing in $\Pi \bar\alpha$, i.e. $\alpha < \lambda \Rightarrow
\cF_\alpha \subseteq \Pi \bar\alpha$ and
$\alpha_1 < \alpha_2 \wedge f_1 \in \cF_{\alpha_1} \wedge f_2 \in
\cF_{\alpha_1} \Rightarrow f_1 <_D f_2$ 
and $\cF_\alpha \ne \emptyset$
for every or at least unboundedly many $\alpha < \lambda$
\underline{then}
$\bigcup\limits_{\alpha < \lambda} \cF_\alpha$ has a common
$<_{D^*_1}$-upper bound.
\end{enumerate}
\medskip

\noindent
[Why?  For each $X \in D$ recall $(\Pi \bar\alpha,<_{D_X})$ has true
cofinality $\lambda_X$ which is regular 
$> \lambda$ hence by \cite[5.7(1A)]{Sh:938}
is pseudo $\lambda^+$-directed
hence there is a common $<_{D_X}$-upper bounded $h_X$ of
$\cup\{\cF_\alpha:\alpha < \lambda\}$.  As we have AC$_{\cP(Y)}$ we
can find a sequence $\langle h_X:X \in D\rangle$ with each $h_X$ as
above.   Define $h \in \Pi \bar\alpha$ by $h(t) = 
\sup\{h_X(t):X \in D\}$, it belongs to $\Pi \bar\alpha$ as
we are assuming $t \in Y \Rightarrow \cf(\alpha_t) \ge
\hrtg(\cP(Y))$.  So $h \in \Pi \bar\alpha$ is a $<_{D_X}$-upper bound
of $\cup\{\cF_\alpha:\alpha < \lambda\}$ for every $X \in D$, hence 
by the choice of $D^*_1$ it is
a $<_{D^*_1}$-upper bound of $\cup\{\cF_\alpha:\alpha < \lambda\}$.]

But by the choice of $D$ in the beginning of the proof we have
$\lambda =$ ps-tcf$(\Pi \bar \alpha,<_D)$ so there is a sequence $\langle
\hat{\cF}_\alpha:\alpha < \lambda\rangle$ witnessing it.   By clauses
(f) + (g) we have $D \subseteq D^*_1$ so 
clearly $\langle \hat\cF_\alpha:\alpha <
\lambda\rangle$ is $<_{D^*_1}$-increasing hence we can apply clause (g) to the
sequence $\langle \hat{\cF}_\alpha:\alpha < \lambda\rangle$ and got 
a $<_{D^*_1}$-upper bound $f \in \Pi\bar \alpha$, contradiction to the
choice of $\langle \hat{\cF}_\alpha:\alpha < \lambda\rangle$ because
$D \subseteq D^*_1$.  So $(*)_2$ really holds.

Choose $X$ as in $(*)_2$, now
\medskip

\noindent
\begin{enumerate}
\item[$(*)_3$]  $D = \text{ dual}(J^{\aleph_1-\comp}_{< \lambda}
[\bar\alpha] + (Y \backslash X))$.
\end{enumerate}
\medskip

\noindent
[Why?  The inclusion $\supseteq$ holds by $(*)_1$ and $(*)_2$, i.e. the
choice of $X$ as a member of $D$.  Now for every
$Z \subseteq X$ which does not belong to 
$J^{\aleph_1-\comp}_{< \lambda}[\bar\alpha]$,
by the definition of $J^{\aleph_1-\comp}_{< \lambda}[\bar\alpha]$ there is an
$\aleph_1$-complete filter $D_Z$ on $Y$ to which $Z$ belongs such that
$\theta :=$ {\rm ps-cf}$(\Pi \bar\alpha,<_D)$ is well defined and $\ge
\lambda$.  But $\theta \ge \lambda^+$ is impossible as we know that $Z
\subseteq X \in J^{\aleph_1\text{-comp}}_{< \lambda^+}[\bar\alpha]$, so 
necessarily $\theta = \lambda$,
hence by the choice of $D$ by using \ref{r16} we have 
$D \subseteq D_Z$, hence $Z \ne \emptyset$
mod $D$.  Together we are done.]
\medskip

\noindent
\begin{enumerate}
\item[$(*)_4$]   $\lambda =$ {\rm ps-tcf}$(\Pi
\bar\alpha,<_{J^{\aleph_1-\comp}_{=\lambda}})$, see clause (B)
of the conclusion of \ref{r18}(2).
\end{enumerate}
\medskip

\noindent
[Why?  By $(*)_3$, the choice of 
$J^{\aleph_1-\comp}_{= \lambda}[\bar\alpha]$ 
and as $\lambda =$ {\rm ps-tcf}$(\Pi \bar \alpha,<_D)$.]
\medskip

\noindent
\begin{enumerate}
\item[$(*)_5$]  $\lambda \notin$ {\rm ps-pcf}$_{\aleph_1-\comp}(\bar\alpha 
\restriction (Y \backslash X))$.
\end{enumerate}
\medskip

\noindent
[Why?  Otherwise there is an $\aleph_1$-complete filter $D'$ on $Y$
such that $Y \backslash X \in D$ and $\lambda = \text{ ps-tcf}(\Pi \bar
\alpha,<_{D'})$.  But this contradicts the choice of $D$ by using \ref{r16}.]

So $X$ is  as required in the desired conclusion of \ref{r18}(2):  
clause (B) by $(*)_4$, clause (C) by $(*)_5$ and clause (A) follows.
Note that the notation
$J^{\aleph_1-\comp}_{=\lambda}[\bar\alpha]$ is justified, as if
$X'$ satisfies the requirements on $X$ then $X' = X$ mod
$J^{\aleph_1-\comp}_{< \lambda}[\bar\alpha]$.
\end{PROOF}

\begin{conclusion}
\label{r19} 
[$\text{AC}_{\cP(Y)}$]  Assume $\bar\alpha \in
{}^Y\text{Ord}$ and each $\alpha_t$ a limit ordinal of cofinality $\ge
\hrtg(\cP(Y))$ and $\ps-\pcf_{\aleph_1-\text{comp}}(\bar\alpha)$ is
not empty.

\noindent
1) If $t \in Y \Rightarrow \cf(\alpha_t) \ge
\hrtg(\text{Fil}^1_{\aleph_1}(Y))$ \then \, there is a function $h$ such that:
\mn
\begin{enumerate}
\item[$\bullet_1$]   the domain of $h$ is $\cP(Y)$
\sn
\item[$\bullet_2$]   the range of $h$ is 
ps-pcf$_{\aleph_1-\comp}(\bar\alpha) \cup \{\mu:\mu = \sup(\mu \cap$
{\rm ps-pcf}$_{\aleph_1-\comp}(\bar\alpha))$ and $\mu$ has cofinality
$\aleph_0$ or just for some $A \in \cP(Y) \backslash
J^{\aleph_1-\text{comp}}_{< \mu}[\bar\alpha]$ there is no
$\aleph_1$-complete filter $D$ on $Y$ such that $A \in D$ and
$(\forall B \in J^{\aleph_1-\text{comp}}_{< \mu}[\bar\alpha])(Y
\backslash B \in D$ and $(\Pi\bar\alpha,<_D)$ has pseudo true
cofinality$\}$, but see $\bullet_5$   
\sn
\item[$\bullet_3$]   $A \subseteq B \subseteq Y \Rightarrow h(A) \le
h(B)$
\sn
\item[$\bullet_4$]   $h(A) = \text{ min}\{\lambda:A \in
J^{\aleph_1-\comp}_{\le \lambda}[\bar\alpha]\}$
\sn
\item[$\bullet_5$]   if $h(A) = \lambda$, cf$(\lambda) > \aleph_0$
then $\lambda \in$ {\rm ps-tcf}$_{\aleph_1-\comp}(\bar\alpha)$, i.e.
for some $\aleph_1$-complete filter $D$ on $Y$ we have $A \in D$ and
ps-tcf$(\Pi \bar\alpha,<_D) = \lambda$
\sn
\item[$\bullet_6$]   the set ps-pcf$_{\aleph_1-\comp}(\bar \alpha)$ has
cardinality $< \hrtg({\cP}(Y))$
\sn
\item[$\bullet_7$]  if $h(A) = \lambda$ and cf$(\lambda) = \aleph_0$
then we can find a sequence $\langle A_n:n < \omega\rangle$ such that
$A = \cup\{A_n:n < \omega\}$ and $h(A_n) < \lambda$ for $n < \omega$
\sn
\item[$\bullet_8$]  $J^{\aleph_1-\comp}_{< \lambda}[\bar\alpha] 
= \{A \subseteq Y:h(A) < \lambda\}$ when cf$(\lambda) > \aleph_0$
\sn
\item[$\bullet_9$]  if cf(otp(ps-pcf$_{\aleph_1-\comp}(\bar\alpha))) >
\aleph_0$ then ps-pcf$_{\aleph_1-\comp}(\bar\alpha)$ has a last member.
\end{enumerate}
\mn
2) Without the extra assumption of part (1), still there is $h$ such
that
\mn
\begin{enumerate}
\item[$\bullet_1$]   $h$ is a function with domain  $\cP(Y)$
\sn
\item[$\bullet_2$]   the range of $h$ is ps-pcf$_{\aleph_1-\comp}
(\bar\alpha) \cup \{\mu:\mu = \sup(\mu \cap$ 
{\rm ps-pcf}$_{\aleph_1-\comp}(\bar\alpha))$ and cf$(\mu) =
\aleph_0$ or just cf$(\mu) < \hrtg(${\rm ps-pcf}$_{\aleph_1-\comp}
(\bar\alpha)$ and $J^{\aleph_1-\comp}_{< \mu}[\bar\alpha] \ne
\cup\{J^{\aleph_1-\comp}_{< \chi}[\bar\alpha]:\chi < \mu\}\}$
\sn
\item[$\bullet_3$]   $A \subseteq B \subseteq Y \Rightarrow h(A) \le h(B)$
\sn
\item[$\bullet_4$]   $h(A) = \text{ min}\{\lambda:A \in
J^{\aleph_1-\comp}_{\le \lambda}[\bar\alpha]\}$
\sn
\item[$\bullet_5$]   if $h(A)$ and cf$(\lambda) \ge \hrtg($ 
{\rm ps-pcf}$_{\aleph_1-\comp}[\bar\alpha])$ \then \,
$\lambda \in$ {\rm ps-pcf}$_{\aleph_1-\comp}(\bar\alpha)$,
i.e. there is an $\aleph_1$-complete filter $D$ on $Y$ such that 
$(\Pi\bar\alpha,<_D)$ has true cofinality $\lambda$
\sn
\item[$\bullet_6$]   as above
\sn
\item[$\bullet_7$]  as above.
\end{enumerate}
\mn
3) In part (1), if also AC$_\alpha$ holds for $\alpha < \hrtg(\cP(Y))$
\then \, we can find a sequence $\langle X_\lambda:\lambda \in$ 
{\rm ps-pcf}$_{\aleph_1-\comp}(\bar\alpha)\rangle$ of subsets of $Y$
such that for every cardinality $\mu,J^{\aleph_1-\comp}_{<
\mu}[\bar\alpha]$ is the $\aleph_1$-complete ideal on $Y$ generated by
$\{X_\lambda:\lambda < \mu$ and $\lambda \in$ 
{\rm ps-pcf}$_{\aleph_1-\comp}(\bar\alpha)\}$. 
\end{conclusion}

\begin{PROOF}{\ref{r19}}   
1) Let $\Theta =$ {\rm ps-pcf}$_{\aleph_1-\comp}(\bar\alpha)$.   
We define function $h$ from ${\cP}(Y)$ into
$\Theta^+$ which is defined as the closure of $\Theta \cup \{0\}$,
i.e. $\Theta \cup \{\mu:\mu = \sup(\mu \cap \Theta)\}$, by 
$h(X) = \text{ Min}\{\lambda \in \Theta^+:X \in 
J^{\aleph_1-\comp}_{\le \lambda}[\bar\alpha]\}$.  It is well defined as
ps-pcf$_{\aleph_1-\comp}(\bar\alpha)$ is a set (see
\cite[5.8(2)]{Sh:938}), non-empty by an assumption 
and $J^{\aleph_1-\comp}_{\le \lambda}
[\bar\alpha] = \cP(Y)$ when $\lambda \ge 
\text{ sup(ps-pcf}_{\aleph_1-\comp}(\bar\alpha))$.  This function,
its range is included in $\Theta^+$, but otp$(\Theta^+)
\le \otp(\Theta) +1$; also clearly $\bullet_1$ of the 
conclusion holds.  Also if $\lambda \in \Theta$ and $X$ is as in
\ref{r18}(2) then $h(X) = \lambda$; so 
$h$ is a function from $\cP(Y)$ into $\Theta^+$
and its range include $\Theta$ hence $|\Theta| < \hrtg(\cP(Y))$ so
$\bullet_6$ holds.  Now first by \ref{r15} we have 
$\theta \in \Theta \backslash \{0\} \Rightarrow
\theta \ge \hrtg(\cP(Y))$, hence $\theta \in \Theta \backslash \{0\}
\Rightarrow \theta > \sup(\Theta \cap \theta)$ so the range of $h$ is
as required in $\bullet_2$.

Second, if $\lambda \in \Theta^+$ 
and cf$(\lambda) = \aleph_0$ then clearly
$\lambda \in \Theta^+ \backslash \Theta$ and we can find an increasing
sequence $\langle \lambda_n:n < \omega\rangle$ of members of
ps-pcf$_{\aleph_1\text{-comp}}(\bar\alpha)$ with limit $\lambda$.  For
each $n$ there is $X_n \in J^{\aleph_1-\comp}_{\le
\lambda_n}[\bar\alpha] \backslash
J^{\aleph_1\text{-comp}}_{<\lambda_n}[\bar\alpha]$ by \ref{r18}(2),
but AC$_{\aleph_0}$ holds, see \ref{r18d} 
hence such a sequence $\langle X_n:n <
\omega\rangle$ exists.  Easily $A := \cup\{X_n:n < \omega\} \in
\cP(Y)$ satisfies $h(A) = \lambda$ hence 
$\lambda \in \text{ Rang}(h)$.   Third, if $\lambda = \sup($
{\rm ps-pcf}$_{\aleph_1-\comp}(\bar\alpha))$ and 
cf$(\lambda) > \aleph_0$, then $\bigcup\limits_{\mu <\lambda}
J_{< \mu}[\bar\alpha] \ne \cP(Y)$ while
$J_{< \lambda^+}(\bar\alpha) = \cP(Y)$
so $h(Y) = \lambda$.

Fourth, assume $\lambda = h(A),\lambda \notin$ 
{\rm ps-pcf}$_{\aleph_1-\comp}(\bar\alpha)$ and cf$(\lambda) >
\aleph_0$, we can find $\langle \lambda_i:i < \cf(\lambda)\rangle$, an
increasing sequence with limit $\lambda$, but
by the definition of $h$ necessarily $\lambda \cap$ 
{\rm ps-pcf}$_{\aleph_1-\comp}(\bar\alpha)$ is an unbounded subset of
$\lambda$ so \wilog \, all are members of
ps-pcf$_{\aleph_1-\comp}(\Pi \bar \alpha)$.
Now $\langle J_i := J^{\aleph_1-\comp}_{< \lambda_i}[\bar\alpha]:i
<  \cf(\lambda)\rangle$ is a $\subseteq$-increasing sequence of
$\aleph_1$-complete ideals on $Y$, no choice is needed, and by 
our present assumption
$\aleph_0 < \cf(\lambda)$ hence the union $J = \cup\{J_i:i <
\text{ cf}(\lambda)\}$ is an $\aleph_1$-complete ideal on $Y$ and
obviously $A \notin J$.  So also $D_1 = \text{ dual}(J) +A$ is an
$\aleph_1$-complete filter hence by \cite[5.9]{Sh:938}
for some $\aleph_1$-complete ideal
$D_2$ extending $D_1$ we have $\mu =$ {\rm ps-tcf}$(\Pi \alpha,<_{D_2})$
is well defined, so by \ref{r18}(2) we have some $D_2 \cap
J^{\aleph_1-\comp}_{\le \mu}[\bar\alpha] \ne \emptyset$ but
$\emptyset = D_2 \cap J_i = D_2 \cap J^{\aleph_1-\comp}_{<
\lambda_1}[\bar\alpha]$ hence $\mu \ge \lambda_i$.  
Hence $\mu \ge \lambda_i$ for every $i < \cf(\lambda)$ 
but $\lambda$ is singular so $\mu > \lambda$ and $\mu
\in$ {\rm ps-pcf}$_{\aleph_1-\comp}(\bar\alpha)$.  Hence $\chi := 
\min($ {\rm ps-pcf}$_{\aleph_1-\comp}(\bar\alpha)
\backslash \lambda)$ is well defined and
$J^{\aleph_1-\comp}_{< \chi}[\bar\alpha] = J$.  But as $h(A) < \chi$ 
we get that $A \in J^{\aleph_1-\comp}_{< \chi}
[\bar\alpha]$, contradiction.

So we have proved $\bullet_5$,
the fifth clause of the conclusion and so 
Rang$(h) \subseteq \ps-\pcf_{\aleph_1-\comp}(\bar\alpha) \cup
\{\mu:\cf(\mu) = \aleph_0$ and $\mu = \sup(\mu \cap$ 
{\rm ps-pcf}$_{\aleph_1-\comp}(\bar\alpha)\}$.
The other clauses follow from the properties of $h$.  

\noindent
2),3)  Similar proof.
\end{PROOF}

\begin{definition}
\label{r20} 
Assume cf$(\mu) < \hrtg(Y)$ and $\mu$ is singular of uncountable
cofinality limit of regulars.  We let

\begin{equation*}
\begin{array}{clcr}

(a) \quad 
\text{ pp}^*_Y(\mu) = \sup\{\lambda:&\text{ for some } \bar\alpha,D 
\text{ we have} \\
  &(a) \quad \lambda =$ {\rm ps-tcf}$(\Pi \bar\alpha,<_D), \\
  &(b) \quad D \text{ is an } \aleph_1-\text{complete filter on } Y \\
  &(c) \quad \bar\alpha = \langle \alpha_t:t \in Y\rangle, 
\text{ each } \alpha_t \text{ regular} \\
  &(d) \quad \mu = \text{ lim}_D \bar \alpha\}
\end{array}
\end{equation*}

\hskip30pt $(b) \quad \text{ pp}^+_Y(\mu) = 
\sup\{\lambda^+:\lambda \text{ as above}\}$.

\hskip30pt $(c) \quad \text{ similarly pp}^*_{\kappa-\comp,Y}(\mu),
\text{pp}^+_{\kappa-\comp,Y}(\mu)$ restricting ourselves 

\hskip50pt to $\kappa$-complete filters $D$; similarly for other properties

\hskip30pt $(d) \quad$ we can replace $Y$ by an $\aleph_1$-complete
filter $D$ on $Y$, this means 

\hskip50pt we fix $D$ but not $\bar\alpha$ above.
\end{definition}

\begin{remark}
\label{r20d}
1) of course, if we consider sets $Y$ such that $\AC_Y$ may fail, it is
   natural to omit the regularity demands, so $\bar \alpha$ is just a
   sequence of ordinals.

\noindent
2) We may use $\bar\alpha$ a sequence of cardinals, not necessarily
   regular; see \S3. 
\end{remark}

\begin{conclusion}
\label{r21}  
[DC + AC$_{\cP(Y)}$]
Assume $\theta = \hrtg(\cP(Y)) < \mu,\mu$ is as in
Definition \ref{r20}, $\mu_0 < \mu$ and $\bar\alpha \in {}^Y(\Reg \cap
\mu^+_0) \wedge \ps-\pcf_{\aleph_1-\comp}(\bar\alpha) \ne \emptyset
\Rightarrow \ps-\pcf_{\aleph_1-\comp}(\bar\alpha) \subseteq \mu$.  If
$\sigma = |\text{Reg } \cap \mu \backslash
\mu_0| < \mu$ and $\kappa = |\text{Reg } \cap 
\text{ pp} ^+_Y(\mu) \backslash \mu_0|$ \then \, 
$\kappa < \hrtg(\theta \times {}^Y \sigma)$.
\end{conclusion}

\begin{remark}
In the ZFC parallel the assumption on $\mu_0 < \mu$ is not necessary.
\end{remark}

\begin{PROOF}{\ref{r21}}
Obvious by Definition \cite[5.6]{Sh:938}
noting Conclusion \ref{r19} above and \ref{r22} below.  That is,
letting $\Lambda = \text{ Reg } \cap \mu \backslash \mu_0$, for every
$\bar\alpha \in {}^Y \Lambda$ by \ref{r20} the set
ps-pcf$_{\aleph_1-\comp}(\bar\alpha)$ is a subset of $\Xi :=
\text{ Reg } \cap \text{\rm pp}^+_Y(\mu) \backslash \mu_0$, and by
claim \ref{r19} it is a set of cardinality $<
\hrtg(\cP(Y))$.  By Claim \ref{r22} below we have $\Xi =
\cup\{\ps-\pcf_{\aleph_1-\comp}(\bar\alpha):\bar\alpha \in
{}^Y \Lambda\}$.  So there is a function $h$ with domain
$\hrtg(\cP(Y)) \times {}^Y \sigma$ such that $\varepsilon <
\hrtg(\cP(Y)) \wedge \bar\alpha \in {}^Y \sigma \Rightarrow
(h(\varepsilon,\bar\alpha)$ is the $\varepsilon$-th member of
ps-pcf$_{\aleph_1-\comp}(\bar\alpha)$ if there is one,
min$(\Lambda)$ otherwise).  So $h$ is a function from $\hrtg(\cP(Y))
\times {}^Y \sigma$ onto the set $\Xi$ of cardinality $\kappa$, so we are done.
\end{PROOF}

\begin{claim}
\label{r22}
\underline{The No Hole Claim}[DC]  

\noindent
1) If $\bar\alpha \in {}^Y\text{\rm Ord}$ 
and $\lambda_2 \in$ {\rm ps-pcf}$_{\aleph_1-\comp}(\bar\alpha)$, for 
transparency $t \in Y \Rightarrow \alpha_t
> 0$ and $\hrtg({\cP}(Y)) \le \lambda_1 = \cf(\lambda_1) < 
\lambda_2$, \then \, for some 
$\bar\alpha' \in \Pi\bar \alpha$ we have $\lambda_1 =$ 
{\rm ps-pcf}$_{\aleph_1-\comp}(\bar\alpha')$.

\noindent
2) In part (1), if in addition AC$_Y$ \then \, \wilog \, $\bar\alpha' \in
{}^Y \Reg$. 

\noindent
3) If in addition {\rm AC}$_{\cP(Y)} + \text{\rm AC}_{< \kappa}$ \then
\, even witnessed by the same filter (on $Y$).
\end{claim}

\begin{PROOF}{\ref{r22}}
1) Let $D$ be an $\aleph_1$-complete filter on $Y$
such that $\lambda_2 =$ {\rm ps-tcf}$(\Pi \bar\alpha,<_D)$, let
$\langle {\cF}_\alpha:\alpha < \lambda_2\rangle$ exemplify this.  

First assume $\hrtg(\text{Fil}^1_{\aleph_1}(Y)) \le \lambda_1$, clearly $f
\in \cF_\alpha \Rightarrow \rk_D(f) \ge \alpha$ for every
$\alpha < \lambda_2$, hence in particular for $\alpha = \lambda_1$ hence
there is $f \in {}^Y\text{Ord}$ such that rk$_D(f) = \lambda_1$ and
now use \cite[5.9]{Sh:938} but there we change the filter $D$, 
(extend it).  In general, i.e. without this extra
assumption, use \ref{r24}(1),(2) below.

\noindent
2) Easy, too.

\noindent
3) Similarly using \ref{r24}(3) below.
\end{PROOF}

\begin{claim}
\label{r24}
Assume $D \in \text{\rm Fil}^1_\kappa(Y),\kappa > \aleph_0,\cF_\alpha
\subseteq {}^Y\text{\rm Ord}$ non-empty 
for $\alpha < \delta$ and $\bar{\cF} = \langle
\cF_\alpha:\alpha < \delta\rangle$ is $<_D$-increasing, $\delta$ a
limit ordinal.

1) {\rm [DC]} There is $f^* \in \Pi \bar\alpha$ which satisfies $f \in 
\cup\{{\cF}_\alpha:\alpha < \lambda_1\} \Rightarrow f <_D f^*$ but there is no
such $f^{**} \in \Pi \bar\alpha$ satisfying $f^{**} <_D f$.

\noindent
2) [DC + AC$_{< \kappa}$]  
For $f^*$ as above, let $D_1 = D_{f^*,\bar{\cF}} :=
\{Y \backslash A:A = \emptyset$ {\rm mod} $D$ or 
$A \in D^+$ and there is $f^{**} \in
{}^Y\text{\rm Ord}$ such that $f^{**} <_{D + A} f^*$ and
$f \in \cup\{{\cF}_\alpha:\alpha < \lambda_1\} \Rightarrow f
<_{D+A} f^{**}\}$.  Now $D_1$ is a $\kappa$-complete filter and
$\emptyset \notin D_1,D_1$ extends $D$ and if {\rm cf}$(\delta) \ge
\hrtg(\cP(Y))$ then $\langle \cF_\alpha:\alpha < \delta\rangle$ 
witness that $f^*$ is a $<_{D_1}$-exact upper bound of $\bar\cF$ hence
$(\prod\limits_{y \in Y} f^*(y),<_{D_1})$ has 
pseudo-true-cofinality {\rm cf}$(\delta)$.

\noindent
3) [{\rm DC + AC}$_{< \kappa} +$ {\rm AC}$_{\cP(Y)}$]

If {\rm cf}$(\delta) \ge \hrtg(\cP(Y))$ \then \, there 
is $f' \in {}^Y\text{\rm Ord}$ which is an $<_{D}$-exact upper
bound of $\bar\cF$, i.e. $f <_D f' \Rightarrow (\exists \alpha <
\delta)(\exists g \in \cF_\alpha)[f < g$ {\rm mod} $D]$ and $f \in
\bigcup\limits_{\alpha < \delta} \cF_\alpha \Rightarrow f <_{D_1,f'}$.
\end{claim}

\begin{PROOF}{\ref{r24}}
1) If not then by DC we can find $\bar f = \langle f_n:n <
\omega\rangle$ such that:
\mn
\begin{enumerate}
\item[$(a)$]  $f_n \in {}^Y\text{Ord}$
\sn
\item[$(b)$]  $f_{n+1} < f_n$ mod $D$
\sn
\item[$(c)$]  if $f \in \bigcup\limits_{\alpha < \delta} \cF_\alpha$ and
$n < \omega$ then $f < f_n$ mod $D$.
\end{enumerate}
\mn
So $A_n = \{t \in Y:f_{n+1}(t) < f_n(t)\} \in D$ hence $\cap\{A_n:n <
\omega\} \in D$, contradiction.

\noindent
2) First, clearly $D_1 \subseteq \cP(Y)$ and by the assumption
$\emptyset \notin D_1$.  Second, if $f^{**}$ witness $A \in D_1$ and $A
\subseteq B \subseteq Y$ then $f^{**}$ witness $B \in D_1$.

Third, we prove $D_1$ is closed under intersection of $< \kappa$
members, so assume $\zeta < \kappa$ and $\bar A = \langle
A_\varepsilon:\varepsilon < \zeta\rangle$ is a sequence of members of
$D_1$.  Let $A := \cap \{A_\varepsilon:\varepsilon <
\zeta\},B_\varepsilon = Y \backslash A_\varepsilon$ for $\varepsilon <
\zeta$ and $B'_\varepsilon = B_\varepsilon \backslash 
\cup\{B_\xi:\xi < \varepsilon\}$ and $B =
\cup\{B_\varepsilon:\varepsilon< \zeta\}$.  Clearly $B = Y \backslash
A,A \subseteq Y$ and $\langle B'_\varepsilon:\varepsilon <
\zeta\rangle$ is a sequence of pairwise disjoint subsets of $Y$ with
union $B$.  But AC$_\zeta$ holds hence we can find $\langle
f^{**}_\varepsilon:\varepsilon < \zeta\rangle$ such that
$f^{**}_\varepsilon$ witness $A_\varepsilon \in D_1$.  Let $f^{**} \in
{}^Y\text{Ord}$ be defined by $f^{**}(t) = 
f^{**}_\varepsilon(t)$ if $t \in B'_\varepsilon$ or $\varepsilon = 0
\wedge t \in Y \backslash B$; easily $B'_\varepsilon \in D^+ \wedge
f \in \bigcup\limits_{\alpha < \delta}
\cF_\alpha \Rightarrow f < f^{**}$ mod $(D + B'_\varepsilon)$ but
$B = \cup\{B'_\varepsilon:\varepsilon < \zeta\}$ and $D$ is $\kappa$-complete
hence $f \in
\bigcup\limits_{\alpha < \delta} \cF_\alpha \Rightarrow f < f^{**}$
mod$(D+B)$.   
So as $A = Y \backslash B$ clearly $f^{**}$ witness $A = 
\bigcap\limits_{\varepsilon < \zeta} A_\varepsilon \in D_1$
so $D_1$ is indeed $\kappa$-complete.

Lastly, assume cf$(\delta) \ge \hrtg(\cP(Y))$ and
we shall show that $f^*$ is an exact upper bound of
$\bar{\cF}$ modulo $D_1$.  So assume $f^{**} \in {}^Y\text{Ord}$ and
$f^{**} < f^*$ mod $D_1$.  

Let $\mathscr{A} = \{A \in D^+_1$: there is $f \in \bigcup\limits_{\alpha <
\delta} \cF_\alpha$ such that $f^{**} \le f$ mod$(D+A)\}$, yes, not $D_1$!
\bigskip

\noindent
\underline{Case 1}:  For every $B  \in D^+_1$ there is $A \in \cA,A
\subseteq B$.

For every $A \in \cA$ let $\alpha_A = \text{ min}\{\beta$: there
 is $f \in \cF_\beta$ such that $f^{**} \le f \text{ mod}(D+A)\}$.

So the sequence $\langle \alpha_A:A \in \cA\rangle$ is well defined.

Let $\alpha(*) = \sup\{\alpha_A +1:A \in \cA\}$, it is $< \delta$ as
cf$(\delta) \ge \hrtg(\cP(Y)) \ge \hrtg(\cA)$.

Now any $f \in \cF_{\alpha(*)}$ is as required because $\{t \in Y:f^{**}(t)
\ge f(t)\}$ contains no $A \in \cA$ hence is $= \emptyset$ mod $D_1$
by the case assumption.
\bigskip

\noindent
\underline{Case 2}:  $B  \in D^+_1$ and there is no $A \in \cA$ such
that $A \subseteq B$.

So $f \in \bigcup\limits_{\alpha < \delta} \cF_\alpha$ implies $\{t
 \in B:f^{**}(t) \le f(t)\} = \emptyset$ mod $D_1$, i.e. $f < f^{**}$
 mod$(D_1+B)$ so by the definition of $D_1$ we have $Y \backslash B
 \in D_1$, contradiction to the case assumption.

\noindent
3) By \cite[5.12]{Sh:938} \wilog \, $\bar{\cF}$ is $\aleph_0$-continuous.
   For every $A \in D^+$ the assumptions hold even if we replace $D$
   by $D+A$ and so there are $D_1,f^*$ as in part (2), we are allowed
   to use part (2) as we have AC$_{< \kappa}$.  As we are
   assuming AC$_{\cP(Y)}$ there is a sequence $\langle (D_A,f_A):A \in
   D^+\rangle$ such that:
\mn
\begin{enumerate}
\item[$(*)_1$]  $(a) \quad D_A$ is a $\kappa$-complete filter
extending $D+A$
\sn
\item[${{}}$]  $(b) \quad f_A \in {}^Y\text{Ord}$ is a $<_{D_A}$-exact
upper bound of $\bar{\cF}$.
\end{enumerate}
\mn
Recall $|A| \le_{\text{qu}} |B|$ is defined as: $A$ is empty or there is a
function from $B$ onto $A$.  Of course, this implies $\hrtg(A) \le \hrtg(B)$.

Let $\bar\cU = \langle \cU_t:t \in Y\rangle$ be defined by $\cU_t =
\{f_A(t):A \in D^+\} \cup \{\sup\{f(t):f \in \bigcup\limits_{\alpha <
\delta} \cF_\alpha\}\}$ hence $t \in Y \Rightarrow 0 < |\cU_t| 
\le_{\text{qu}} \cP(Y)$
even uniformly (or recall that we have AC$_Y$) so there is a sequence $\langle
h_t:t \in Y\rangle$ such that $h_t$ is a function from $\cP(Y)$ onto
$\cU_t$ hence $|\prod\limits_{t \in Y} \cU_t| \le_{\text{qu}} \cP(Y) \times Y
\le_{\text{qu}} \cP(Y \times Y)$ but
AC$_{\cP(Y)}$ holds  hence $Y$ can be well ordered but \wilog \, $Y$
is infinite hence $|Y \times Y| = Y$, so $|\prod\limits_{t \in Y}
\cU_t| \le_{\text{qu}} |\cP(Y)|$.

Let $\cG = \{g:g \in \prod\limits_{t \in Y} \cU_t$ and not for every $f
\in \bigcup\limits_{\alpha < \delta} \cF_\alpha$ do we have $f < g$ mod
$D\}$, so $|\cG| \le |\prod\limits_{t \in Y} \cU_t| \le_{\text{qu}} 
|\cP(Y \times Y)| = |\cP(Y)|$ hence $\hrtg(\cG) \le \text{ cf}(\delta)$.

Now for every $g \in \cG$ the sequence $\langle\{\{t \in Y:g(t) \le
f(t)\}:f \in \bigcup\limits_{\beta < \alpha} \cF_\beta\}:\alpha <
\delta\rangle$ is a $\subseteq$-increasing sequence of subsets of
$\cP(Y)$, but $\hrtg(\cP(Y)) \le \text{ cf}(\delta)$ hence the
sequence is eventually constant and let $\alpha(g) < \delta$ be
minimal such that
\mn
\begin{enumerate}
\item[$(*)_g$]  $(\forall \beta)[\alpha(g) \le \beta < \delta
\Rightarrow \{\{t \in Y:g(t) \le f(t)\}:f \in \bigcup\limits_{\gamma <
\beta} \cF_\gamma\} = \{\{t \in Y:g(t) \le f(t)\}:f \in
\bigcup\limits_{\gamma < \alpha} \cF_\gamma\}]$.
\end{enumerate}
\mn
But recalling $\hrtg(\cG) \le \text{ cf}(\delta)$, the ordinal
$\alpha(*) := \sup\{\alpha(g):g \in \cG\}$ is $< \delta$.   Now choose
$f^* \in \cF_{\alpha(*)+1}$ and define $g^* \in \prod\limits_{t \in Y}
\cU_t$ by $g^*(t) = \text{ min}(\cU_t \backslash f^*(t))$, well
defined as sup$\{f(t):t \in \bigcup\limits_{\alpha < \delta}
\cF_\alpha\} \in \cU_t$.  It is easy to check that $g^*$ is as required.
\end{PROOF}

\begin{observation}
\label{r25}
Let $D$ be a filter on $Y$.

\noindent
If $D$ is $\kappa$-complete for every $\kappa$ \then \, for every
   $f \in {}^Y\text{\rm Ord}$ and $A \in D^+$ there is $B \subseteq A$
   from $D^+$ such that $f \rest B$ is constant.
\end{observation}

\begin{PROOF}{\ref{r25}}
Straight.
\end{PROOF}
\newpage

\section {Composition and generating sequence of pseudo pcf}

How much choice suffice to show $\lambda = \ptf(\prod\limits_{(i,j) \in Y}
\lambda_{i,j}/D)$ when $\lambda_i$ is the pseudo true equality of
$\prod\limits_{j \in Y_i} \lambda_{i,j}$ for $i \in Z$ where
$Z = \{i:(i,j) \in Y\}$ and $Y_i = \{(i,j):i \in Z,j \in Y_i\}$ 
and $\lambda = \text{ ps-tcf}(\prod\limits_{i \in
Z} \lambda_i,<_E)$?  This is \ref{e3}, the parallel of
\cite[Ch.II,1.10,pg.12]{Sh:g}.

\begin{claim}
\label{e1}
If $\boxplus$ below holds \then \, for some partition $(Y_1,Y_2)$ of $Y$
and club $E$ of $\lambda$ we have
\mn
\begin{enumerate}
\item[$\oplus$]   $(a) \quad$ if $Y_1 \in D^+$ and $f,g \in
\cup\{\cF_\alpha:\alpha \ge \text{\rm min}(E)\}$ \then \,
$f=g$ {\rm mod}$(D + Y_1)$
\sn
\item[${{}}$]   $(b) \quad$ if $Y_2 \in D^+$ then $\langle \cF_\alpha:
\alpha \in E \rangle$ is $<_{D+Y_2}$-increasing
\end{enumerate}
\mn
where
\mn
\begin{enumerate}
\item[$\boxplus$]   $(a) \quad \lambda$ is regular $\ge \hrtg(\cP(Y)))$
\sn
\item[${{}}$]   $(b) \quad \cF_\alpha \subseteq {}^Y\text{\rm Ord}$ for
$\alpha < \lambda$ is non-empty
\sn
\item[${{}}$]   $(c) \quad$ if $\alpha_1 < \alpha_2 < \lambda$ and
$f_\ell \in \cF_{\alpha_\ell}$ for $\ell=1,2$ \then \,
$f_1 \le f_2$ {\rm mod} $D$.
\end{enumerate}
\end{claim}

\begin{PROOF}{\ref{e1}}
For $Z \in D^+$ let
\mn
\begin{enumerate}
\item[$(*)_1$]   $(a) \quad S_Z = \{(\alpha,\beta):\alpha \le \beta <
\lambda$ and for some $f \in \cF_\alpha$ and $g \in \cF_\beta$ we have

\hskip25pt $f < g$ mod $(D + Z)\}$
\sn
\item[${{}}$]  $(b) \quad S^+_Z = \{(\alpha,\beta):\alpha \le 
\beta < \lambda$ and for every $f \in \cF_\alpha$ and $g \in 
\cF_\beta$ we have

\hskip25pt  $f < g$ mod $(D+Z)\}$.
\end{enumerate}
\mn
Note
\mn
\begin{enumerate}
\item[$(*)_2$]   $(a) \quad$ if $\alpha_1 \le \alpha_2 \le \alpha_3
\le \alpha_4$ and $(\alpha_2,\alpha_3) \in S_Z$ \ then \,
$(\alpha_1,\alpha_4) \in S_Z$
\sn
\item[${{}}$]   $(b) \quad$ similarly for $S^+_Z$
\sn
\item[${{}}$]  $(c) \quad$ if $\alpha_1 \le \alpha_2 \le \alpha_3
\le \alpha_4$ and $(\alpha_1 \ne \alpha_2) \wedge (\alpha_3 \ne
\alpha_4)$ and $(\alpha_2,\alpha_3) \in S_Z$

\hskip25pt then $(\alpha_1,\alpha_4) \in S^+_Z$
\sn
\item[${{}}$]  $(d) \quad S_Z \subseteq S^+_Z$.
\end{enumerate}
\mn
[Why?  By the definitions.]

Let
\mn
\begin{enumerate}
\item[$(*)_3$]   $J := \{Z \subseteq Y:Z \in$ {\rm dual}$(D)$ or $Z
\in D^+$ and $(\forall^\infty \alpha < \lambda)(\exists \beta)
((\alpha,\beta) \in S^+_Z)$.
\end{enumerate}
\mn
Next
\mn
\begin{enumerate}
\item[$(*)_4$]   $(a) \quad J$ is an $\aleph_1$-complete ideal on $Y$
\sn
\item[${{}}$]  $(b) \quad$ if $D$ is $\kappa$-complete then $J$ is
$\kappa$-complete\footnote{note that AC$_\kappa$ holds in the non-trivial
case as AC$_{\cP(Y)}$ holds, see \ref{r25}}
\sn
\item[${{}}$]  $(c) \quad J = \{Z \subseteq Y:Z \in \text{ dual}(D)$
or $Z \in D^+$ and $(\forall^\infty \alpha < \lambda)(\exists \beta)$

\hskip25pt $((\alpha,\beta) \in S_Z)\}$.
\end{enumerate}
\mn
[Why?  For clauses (a),(b) check and for clause (c) recall $(*)_2(c)$.]

Let
\mn
\begin{enumerate}
\item[$(*)_5$]   $(a) \quad$ for $Z \in J^+$ let $\alpha(Z) = \text{
min}\{\alpha < \lambda$: for no $\beta \in (\alpha,\lambda)$ do we
have

\hskip25pt  $(\alpha,\beta) \in S_Z\}$
\sn
\item[${{}}$]   $(b) \quad \alpha(*) = \sup\{\alpha_Z:Z \in J^+\}$
\sn
\item[$(*)_6$]  $(a) \quad$ for $Z \in J^+$ we have $\alpha(Z) <
\lambda$
\sn
\item[${{}}$]  $(b) \quad \alpha(*) < \lambda$.
\end{enumerate}
\mn
[Why?  Clause (a) by the definition of the ideal $J$, 
and clause (b) as $\lambda = \text{ cf}(\lambda) \ge \hrtg(\cP(Y))$.]

Let
\mn
\begin{enumerate}
\item[$(*)_7$]   $(a) \quad$ for $Z \in D^+$ let $f_Z:\lambda
\rightarrow \lambda +1$ be defined by $f_Z(\alpha) =$

\hskip25pt  $\text{Min}\{\beta:(\alpha,\beta) \in S^+_Z$ or $\beta = \lambda\}$
\sn
\item[${{}}$]   $(b) \quad f_*:\lambda \rightarrow \lambda$ be defined
by: $f_*(\alpha) = \sup\{f_Z(\alpha):Z \in D^+ \cap J\}$
\sn
\item[${{}}$]  $(c) \quad E = \{\delta:\delta$ a limit ordinal $<
\lambda$ such that $\alpha < \delta \Rightarrow f_*(\alpha) < \delta\}
\backslash \alpha(*)$.
\end{enumerate}
\mn
Hence
\mn
\begin{enumerate}
\item[$(*)_8$]  $(a) \quad$ if $Z \in D^+ \cap J$ then
$f_Z$ is indeed a function from $\lambda$ to
$\lambda$
\sn
\item[${{}}$]  $(b) \quad f_*$ is indeed a function from $\lambda$ to
$\lambda$
\sn
\item[${{}}$]  $(c) \quad E$ is a club of $\lambda$.
\end{enumerate}
\mn
[Why?  Clause (a) by the definition of $J$ and clause (b) as $\lambda
= \text{ cf}(\lambda) \ge \hrtg(\cP(Y))$ and clause (c) follows from (b).]
\mn
\begin{enumerate}
\item[$(*)_9$]   Let $\alpha_0 = \text{ min}(E),\alpha_1 = \text{
min}(E \backslash (\alpha_0 +1))$ choose $f_0 \in \cF_{\alpha_0},f_1
\in \cF_{\alpha_1}$ and let $Y_1 = \{y \in Y:f_0(y) = f_1(y)\}$ and $Y_2 =
Y \backslash Y_1$
\sn
\item[$(*)_{10}$]   $(Y_1,Y_2,E)$ are as required.
\end{enumerate}
\mn
[Why?  Think.]
\end{PROOF}

\begin{claim}
\label{e2}
[{\rm AC}$_Y$]

We have $\lambda =$ {\rm ps-tcf}$(\bar\alpha_1,<_D)$ \when \,
\mn
\begin{enumerate}
\item[$(a)$]  $\bar\alpha \in {}^Y\text{\rm Ord}$ and $t \in Y
\Rightarrow \cf(\alpha_t) \ge \hrtg(Y)$
\sn
\item[$(b)$]  $E$ is the equivalence relation on $Y$ such that $s E t
\Leftrightarrow \alpha_s = \alpha_t$
\sn
\item[$(c)$]  $\lambda =$ {\rm ps-tcf}$(\Pi\bar\alpha,<_D)$
\sn
\item[$(d)$]  $Y_1 = Y/E$
\sn
\item[$(e)$]  $D_1 = \{Z \subseteq Y/E:\cup\{X:X \in Z\} \in D\}$
\sn
\item[$(f)$]  $\bar\alpha_1 = \langle \alpha_{1,y_1}:y_1 \in Y_1
\rangle$ where $y_1 = y/E \Rightarrow \alpha_{1,y_1} = \alpha_y$.
\end{enumerate}
\end{claim}

\begin{remark}
But see \ref{e23} which eliminates $\AC_Y$.
\end{remark}

\begin{PROOF}{\ref{e2}}
Let $\bar{\cF} = \langle \cF_\alpha:\alpha < \lambda\rangle$ witness
$\lambda = \text{ ps-tcf}(\bar\alpha,<_D)$.  For $f \in
\cup\{\cF_\alpha:\alpha < \lambda\}$ let $f^{[*]} \in 
{}^Y\text{Ord}$ be defined by $f^{[*]}(s) = \sup\{f(t):t \in s/E\}$.
Clearly $f^{[*]} \in \Pi \bar\alpha$ as $t \in Y \Rightarrow \text{
cf}(\alpha_t) \ge \hrtg(Y)$ by clause (a) of the assumption.  Let
$\cF^{[*]}_\alpha = \{f^{[*]}:f \in \cF_\alpha\}$ for $\alpha <
\lambda$ so $\langle \cF^{[*]}_\alpha:\alpha < \lambda\rangle$ exists and
$\cF^{[*]}_\alpha \subseteq \Pi\bar\alpha$.  Also $f_1 \in
\cF_{\alpha_1} \wedge f_2 \in \cF_{\alpha_2} \wedge \alpha_1 <
\alpha_2 < \lambda \Rightarrow f_1 \le_D f_2 \Rightarrow f^{[*]}_1
\le_D f^{[*]}_2$ hence $\alpha_1 < \alpha_2 < \lambda \wedge f_2 \in
\cF^{[*]}_1 \wedge f_2 \in \cF^{[*]}_2 \Rightarrow f_1 \le_D f_2$.

Now apply \ref{e1}, getting $(Y'_1,Y'_2)$ as there, but by the choice
of $\bar{\cF}$ necessarily $Y'_1 = \emptyset$ mod $D$.  Hence 
for some club $E$ of $\lambda,\langle
\cF^{[*]}_\alpha:\alpha \in E\rangle$ is $<_D$-increasing cofinal in
$\Pi\bar\alpha$.

Lastly, for $f \in \cup\{\cF^{[*]}_\alpha:\alpha \in E\}$ let
$f^{[**]} \in {}^{(Y_1)}\text{Ord}$ be defined by $f^{[**]}(t/E) =
f(t)$, well defined as $f \rest (t/E)$ is constant.  Let
$\cF^{[**]}_\alpha := \{f^{[**]}:f \in \cF^{[*]}_\alpha\}$ for $\alpha
\in E$.  Easily $\langle \cF^{[**]}_\alpha:\alpha \in E\rangle$
witness the desired conclusions.
\end{PROOF}

\noindent
By the following claims we do not really lose by using $\ga \subseteq
\Reg$ instead $\bar\alpha \in {}^Y \Ord$ as by \ref{e24}, \wilog \,
$\alpha_t = d(\alpha_t)$ (when $\AC_Y$) and by \ref{e23} (or \ref{e4}
of $\AC_Y$). 
\begin{claim}
\label{e23}
Assume $\bar\alpha \in {}^Y${\rm Ord}, $D \in$ {\rm Fil}$(Y)$ and
$\lambda =$ {\rm ps-pcf}$(\Pi \bar\alpha,<_D)$ so $\lambda$ is
regular and $y \in Y \Rightarrow \alpha_y < \lambda$.

If $\langle \cF_\alpha:\alpha < \lambda\rangle$ witness $\lambda =$
{\rm ps-tcf}$(\Pi \bar\alpha,<_D)$ and $y \in Y \Rightarrow \cf(\alpha_y) 
\ge \hrtg(Y)$ and  $\lambda \ge \hrtg(Y)$ \then \, for some $e$:
\mn
\begin{enumerate}
\item[$(a)$]  $e \in \text{\rm eq}(Y) = \{e:e$ an equivalence relation
on $Y\}$
\sn
\item[$(b)$]   the sequence $\cF_e = \langle \cF_{e,\alpha}:\alpha <
\lambda\rangle$ witness {\rm ps-tcf}$(\langle \alpha_{y/e}:y \in
Y/e\rangle,D/e\rangle$ where
\sn
\item[$(c)$]  $\alpha_{y/e} = \alpha_y,D/e = \{A/e:A \in D\}$ where 
$A/e = \{y/e:y \in A\}$ and $\cF_{e,\alpha} = \{f^{[*]}:f \in
\cF_\alpha\},f^{[*]}:Y/e \rightarrow \Ord$ is defined by $f^{[*]}(t/e) =
\sup\{f(s):s \in t/e\}$; noting $\hrtg(Y/e) \le \hrtg(Y)$.
\end{enumerate}
\end{claim}

\begin{PROOF}{\ref{e23}}
Let $e = \text{ eq}(\bar\alpha) = \{(y_1,y_2):y_1 \in Y,y_2 \in Y$ and
$\alpha_{y_1} = \alpha_{y_2}\}$.  For each $f \in \Pi \bar\alpha$ let
the function $f^{[*]} \in \Pi \bar\alpha$ be defined by $f^{[*]}(y) =
\sup\{f(z):z \in y/e\}$.  Clearly $f^{[*]}$ is a function from
$\prod\limits_{y \in Y} (\alpha_y +1)$ and it belongs to $\Pi
\bar\alpha$ as $y \in Y \Rightarrow \text{ cf}(\alpha_y) \ge
\hrtg(y)$.  Let $H:\lambda \rightarrow \lambda$ be: $H(\alpha) =
\text{ min}\{\beta < \lambda:\beta > \alpha$ and there are $f_1 \in
\cF_\alpha$ and $f_2 \in \cF_\beta$ such that $f^{[*]}_1 < f_2$ mod $D\}$,
well defined as $\cF$ is cofinal in $(\Pi \bar\alpha,<_D)$.  We
choose $\alpha_i < \lambda$ by induction on $i$ by: $\alpha_i =
\cup\{H(\alpha_j) +1:j < i\}$.  So $\alpha_0 = 0$ and $\langle
\alpha_i:i < \lambda\rangle$ is increasing continuous.  Let $\cF'_i =
\{f^{[*]}:f \in g_{\alpha_i}$ and there is $g \in \cF_{H(\alpha_i)} =
\cF_{\alpha_{i+1}-1}$ such that $f < g$ mod $D\}$. 

So
\mn
\begin{enumerate}
\item[$(*)_1$]  $\cF'_i \subseteq \Pi\{f \in \Pi
\bar\alpha:\text{eq}(f)$ refine $\text{eq}(\bar\alpha)\}$.
\end{enumerate}
\mn
[By its choice]
\mn
\begin{enumerate}
\item[$(*)_2$]  $\cF'_i$ is non-empty.
\end{enumerate}
\mn
[Why?  By the choice of $H(\alpha_i)$.]
\mn
\begin{enumerate}
\item[$(*)_3$]  if $i(1) < i(2) < \lambda$ and $h_\ell \in
\cF'_{i_\ell}$ for $\ell=1,2$ then $g_1 < g_2$ mod $D$.
\end{enumerate}
\mn
[Why?  Let $g_1 \in F_{H(\alpha_{i(1)}})$ be such that 
$h^{[*]}_1 \le g_1$ mod $D$, exists by the definition of
$\cF^1_{i(1)}$.  But $H(\alpha_{i(1)}) < 
\alpha_{i(1)+1} \le \alpha_{i(2)}$ hence 
$g_1 \le h_2$ mod $D$ so together we are done.]
\mn
\begin{enumerate}
\item[$(*)_4$]   $\bigcup\limits_{i < \lambda} \cF'_i$ is cofinal in
$(\Pi \bar\alpha,<_D)$.
\end{enumerate}
\mn
[Easy, too.]

Lastly, let $\cF^+_i = \{f/e:e \in \cF'_i\}$ where $f/e \in
{}^{Y/e}\text{Ord}$, is defined by $(f/e)(y/e) = f(y)$, clearly well
defined.
\end{PROOF}

\begin{claim}
\label{e29}
Assume {\rm AC}$_Y$ and $\bar\alpha_\ell = \langle \alpha^\ell_y:y
   \in Y\rangle \in {}^Y${\rm Ord} for $\ell=1,2$. If $y \in Y
   \Rightarrow \cf(\alpha^1_y) = \cf(\alpha^2_y)$
 \then \, $\lambda =$ {\rm ps-tcf}$(\Pi \bar\alpha_1,<_D)$ iff $\lambda =
\ps-\tcf(\Pi \bar\alpha_2,<_D)$.
\end{claim}

\begin{proof}
Straight.
\end{proof}

\noindent
Now we come to the heart of the matter
\begin{theorem}
\label{e3}
\underline{The Composition Theorem}
[Assume {\rm AC}$_Z$ and $\kappa \ge \aleph_0$]

We have $\lambda =$ {\rm ps-tcf}$(\prod\limits_{(i,j) \in Y}
\lambda_{i,j},<_D)$ and $D$ is a $\kappa$-complete filter on $Y$ \when \,:
\mn
\begin{enumerate}
\item[$(a)$]  $E$ is a $\kappa$-complete filter on $Z$
\sn
\item[$(b)$]  $\langle \lambda_i:i \in Z\rangle$ is a sequence of
regular cardinals
\sn
\item[$(c)$]  $\lambda =$ {\rm ps-tcf}$(\prod\limits_{i \in Z}
\lambda_i,<_E)$
\sn
\item[$(d)$]  $\bar Y = \langle Y_i:i \in Z\rangle$
\sn
\item[$(e)$]  $\bar D = \langle D_i:i \in Z\rangle$
\sn
\item[$(f)$]  $D_i$ is a $\kappa$-complete filter on $Y_i$
\sn
\item[$(g)$]  $\bar\lambda = \langle \lambda_{i,j}:i \in Z,j \in
Y_i\rangle$ is a sequence of regular cardinals (or just limit ordinals)
\sn
\item[$(h)$]  $\lambda_i =$ {\rm ps-tcf}$(\prod\limits_{j \in Y_i}
\lambda_{i,j},<_{D_i})$
\sn
\item[$(i)$]  $Y = \{(i,j):j \in Y_i$ and $i \in Z\}$
\sn
\item[$(j)$]  $D = \{A \subseteq Y$: for some $B \in E$ we have $i \in
B \Rightarrow \{j:(i,j) \in A\} \in D_i\}$.
\end{enumerate}
\end{theorem}

\begin{PROOF}{\ref{e3}}
\mn
\begin{enumerate}
\item[$(*)_0$]  $D$ is a $\kappa$-complete filter on $Y$.
\end{enumerate}
\mn
[Why?  Straight (and do not need any choice).]

Let $\langle \cF_{i,\alpha}:\alpha < \lambda_i,i \in Z\rangle$ be such
that
\mn
\begin{enumerate}
\item[$(*)_1$]  $(a) \quad \bar{\cF}_i = \langle \cF_{i,\alpha}:\alpha <
\lambda_i\rangle$ witness $\lambda_i =$ {\rm ps-tcf}$(\prod\limits_{j
\in y_i} \lambda_{i,j},<_{D_i})$
\sn
\item[${{}}$]  $(b) \quad \cF_{i,\alpha} \ne \emptyset$.
\end{enumerate}
\mn
[Why?  Exists by clause (h) of the assumption and AC$_Z$, for clause
(b) recall \cite[5.6]{Sh:938}.]

By clause (c) of the assumption let $\bar{\cG}$ be such that
\mn
\begin{enumerate}
\item[$(*)_2$]  $(a) \quad \bar{\cG} = \langle \cG_\beta:
\beta < \lambda \rangle$
witness $\lambda =$ {\rm ps-tcf}$(\prod\limits_{i \in Z}
\lambda_i,<_E)$
\sn
\item[${{}}$]  $(b) \quad \cG_\beta \ne \emptyset$ for $\beta < \lambda$.
\end{enumerate}
\mn
Now for $\beta < \lambda$ let
\mn
\begin{enumerate}
\item[$(*)_3$]  $\cF_\beta := \{f:f \in \prod\limits_{(i,j) \in Y}
\lambda_{i,j}$ and for some $g \in \cG_\beta$ and $\bar h = \langle
h_i:i \in Z  \rangle \in \prod\limits_{i \in Z} \cF_{i,g(i)}$ we have
$(i,j) \in y \Rightarrow f((i,j)) = h_i(j)\}$
\sn
\item[$(*)_4$]  the sequence $\langle \cF_\beta:\beta < \lambda\rangle$
is well defined (so exists)
\sn
\item[$(*)_5$]  if $\beta_1 < \beta_2,f_1 \in \cF_{\beta_1}$ and $f_2
\in \cF_{\beta_2}$ \then \, $f^*_1 <_D f^*_2$.
\end{enumerate}
\mn
[Why?  Let $g_\ell \in \cG_{\beta_\ell}$ and $\bar h_\ell = \langle
h^\ell_i:i \in Z \rangle \in \prod\limits_{i \in Z} \cF_{i,g_\ell(i)}$,
witness $f_\ell \in \cF_{\beta_\ell}$ for $\ell=1,2$.  As $\beta_1 <
\beta_2$ by $(*)_2$ we have $B := \{i \in Z:g_1(i) < g_2(i)\} \in
E$.  For each $i \in B$ we know that $g_1(i) < g_2(i) < 
\lambda_i$ and as $h^1_i \in \cF_{i,g_i(i)},h^2_i \in \cF_{i,g_2(i)}$,
recalling the choice of $\langle \cF_{i,\alpha}:\alpha < \lambda_i
\rangle$, see $(*)_1$, we have $A_i \in D_i$ where for every $i \in Z$
we let $A_i := \{j \in Y_i:h^1_i(j) < h^2_i(j)\}$.  As 
$\bar h_1,\bar h_2$ exists clearly $\langle
A_i:i \in Z\rangle$ exist hence $A = \{(i,j):i \in B$ and $j \in
A_i\}$ is a well defined subset of $Y$ and it belongs to $D$ by the
definition of $D$.

Lastly $(i,j) \in A \Rightarrow f_1((i,j)) < f_2((i,j))$ shown above;
so by the definition of $D$ we are done.]
\mn
\begin{enumerate}
\item[$(*)_6$]   for every $\beta < \lambda$ the set $\cF_\beta$ is
non-empty.
\end{enumerate}
\mn
[Why?  Recall $\cG_\beta \ne \emptyset$ by $(*)_2$ and let $g \in
\cG_\beta$.  As $\langle \cF_{i,g(i)}:i \in Z\rangle$ is a sequence of
non-empty sets, and we are assuming AC$_Z$ there is a sequence
$\langle h_i:i \in Z\rangle \in \prod\limits_{i \in Z} \cF_{i,g(i)}$.  Let $f$
be the function with domain $Y$ defined by $f((i,j)) = h_i(j)$; so
$g,\bar h$ witness $f \in \cF_\beta$, so 
$\cF_\beta \ne \emptyset$ as required.]
\mn
\begin{enumerate}
\item[$(*)_7$]   if $f_* \in \prod\limits_{(i,j) \in Y} \lambda_{i,j}$
\then \, for some $\beta < \lambda$ and $f \in \cF_\beta$ we have
\newline
 $f_* < f$ mod $D$.
\end{enumerate}
\mn
[Why?  We define $\bar f = \langle f^*_i:i \in Z\rangle$ as follows:
$f^*_i$ is the function with domain $Y_i$ such that

\[
j \in Y_i \Rightarrow f^*_i(j) = f((i,j)).
\]

\mn
Clearly $\bar f$ is well defined and for each $i,f^*_i \in \prod_{j
\in Y_i} \lambda_{i,j}$ hence by $(*)_1(a)$ for some $\alpha <
\lambda_i$ and $h \in \cF_{i,\alpha}$ we have $f^*_i < h$ mod $D_i$
and let $\alpha_i$ be the first such $\alpha$ so $\langle \alpha_i:i \in Z
\rangle$ exists.  

By the choice of $\langle \cG_\beta:\beta < \lambda\rangle$ there are
$\beta < \lambda$ and $g \in \cG_\beta$ such that $\langle \alpha_i:i
\in Z\rangle < g$ mod $E$ hence $A := \{i \in Z:\alpha_i < g(i)\}$
belongs to $E$.
So $\langle \cF_{i,g(i)}:i \in Z\rangle$ is a sequence of non-empty sets
hence recalling AC$_Z$ there is a sequence $\bar h = \langle h_i:i \in
Z \rangle \in \prod\limits_{i \in Z} \cF_{i,g(i)}$.  By the property of
$\langle \cF_{i,\alpha}:\alpha < \lambda_i\rangle$, we have $i \in A
\Rightarrow f^*_i < h_i$ mod $D_i$.

Lastly, let $f \in \prod\limits_{(i,j) \in Y} \lambda_{i,j}$ be defined by
$f((i,j)) = h_i(j)$.  Easily $g,\bar h$ witness that $f \in \cF_\beta$,
and by the definition of $D$, recalling $A \in E$ and the choice of
$\bar h$ we have $f_* < f$ mod $D$, so we are done.]

Together we are done proving the theorem.
\end{PROOF}

\begin{conclusion}
\label{e5}
\underline{The pcf closure conclusion}
Assume AC$_{\cP(\ga)}$.  We have $\gc =$ {\rm ps-pcf}$_{\aleph_1-\comp}(\gc)$ 
\when \,:
\mn
\begin{enumerate}
\item[$(a)$]   $\ga$ a set of regular cardinals
\sn
\item[$(b)$]   $\hrtg(\cP(\ga)) < \text{ min}(\ga)$
\sn
\item[$(c)$]   $\gc =$ {\rm ps-pcf}$_{\aleph_1-\comp}(\ga)$.
\end{enumerate}
\end{conclusion}

\begin{PROOF}{\ref{e5}}
Assume $\lambda \in$ {\rm ps-pcf}$_{\aleph_1-\comp}(\gc)$, hence
there is $E$ an $\aleph_1$-complete filter on $\gc$ such that $\lambda
=$ {\rm ps-tcf}$(\Pi \gc,<_E)$.  As we have AC$_{\cP(\ga)}$ by
\ref{r16} (as the $D$ there is unique) 
there is a sequence $\langle D_\theta:\theta \in
\gc\rangle,D_\theta$ an $\aleph_1$-complete filter on $\ga$ such that
$\theta =$ {\rm ps-tcf}$(\Pi \ga,<_{D_\theta})$, also by \ref{r19}
there is a function $h$ from $\cP(\ga)$ onto $\gc$, let $E_1 = \{S
\subseteq \cP(\ga):\{\theta \in \gc:h^{-1}\{\theta\} \subseteq S\} 
\in E\}$.  Similarly to \ref{e2} with
$\cP(Y)$ here standing for $Y$ there, we have $\lambda =$ {\rm ps-tcf}
$(\Pi\{h(A):A \in \cP(\ga)\},<_{E_1})$ and $E_1$ is an
$\aleph_1$-complete filter on $\cP(\ga)$.

Now we apply Theorem \ref{e3} with $E_1,\langle D_{h(\gb)}:\gb \in
\cP(\ga)\rangle,\lambda,\langle h(\gb):\gb \in \cP(\ga)\rangle,\langle
\theta:\theta \in \ga\rangle$ here standing for $E,\langle D_i:i \in
Z\rangle,\lambda,\langle \lambda_i:i \in Z \rangle,\langle
\lambda_{i,j}:j \in Y_i\rangle$ for every $j \in Z$ (constant here).
We get a filter $D_1$ on $Y = \{(\gb,\theta):\gb \in \cP(\ga),\theta
\in \ga\rangle$ such that $\lambda = \ps-\tcf(\Pi\{\theta:(b,\theta) \in
Y\},<_{D_1})$.

Now $|Y| = |\cP(\ga)|$ as $\aleph_0 \le |\ga|$ or $\ga$ finite and all
is trivial so applying \ref{e2} again we get an $\aleph_1$-complete
filter $D$ on $\ga$ such that $\lambda =$ {\rm ps-tcf}$(\Pi \ga,<_D)$,
so we are done.
\end{PROOF}

\begin{theorem}
\label{e7}
Assume {\rm AC}$_{\gc}$ and {\rm AC}$_{\cP(\ga)}$.
\Then \, $\gc =$ {\rm ps-pcf}$_{\aleph_1-\comp}(\gc)$ has 
a closed generating sequence for
$\aleph_1$-complete filters (see below) \when \,:
\mn
\begin{enumerate}
\item[$(a)$]   $\ga$ is a set of regular cardinals
\sn
\item[$(b)$]   $\hrtg(\cP(\ga)) < \text{\rm min}(\ga)$
\sn
\item[$(c)$]   $\gc =$ {\rm ps-pcf}$_{\aleph_1-\comp}(\ga)$.
\end{enumerate}
\end{theorem}

\begin{definition}
\label{e10}
For a set $\ga$ of regular cardinals.

\noindent
1) We say $\bar{\gb} = \langle \gb_\lambda:\lambda \in
\text{\rm ps-pcf}_{\aleph_1-\text{\rm com}}(\ga)\rangle$ 
is a generating sequence for $\ga$ \when \,:
\mn
\begin{enumerate}
\item[$(\alpha)$]   $\gb_\lambda \subseteq \ga \subseteq \gc \subseteq$
{\rm ps-pcf}$_{\aleph_1-\comp}(\ga)$
\sn
\item[$(\beta)$]   $J_{< \lambda}[\ga]$ is the $\aleph_1$-complete
ideal in $\ga$ generated by $\{\gb_\theta:\theta \in 
\pcf_{\aleph_1-\comp}(\ga)$ and $\theta < \lambda\}$.
\end{enumerate}
\mn
2) We say $\bar{\cF}$ is a witness for $\bar{\gb} = \langle
\gb_\lambda:\lambda \in \gc \subseteq$ 
{\rm ps-pcf}$_{\aleph_1-\comp}(\ga)\rangle$ \when \, :
\mn
\begin{enumerate}
\item[$(\alpha)$]   $\bar{\cF} = \langle \bar{\cF}_\lambda:\lambda \in
\gc\rangle$
\sn
\item[$(\beta)$]   $\bar{\cF}_\lambda = \langle
\cF_{\lambda,\alpha}:\alpha < \lambda \rangle$ witness {\rm ps-tcf}$(\Pi
\ga,<_{J_{=\lambda}[\ga]})$.
\end{enumerate}
\mn
3) Above $\bar{\gb}$ is closed \when \, $\gb_\lambda = \ga \cap$ 
{\rm ps-pcf}$_{\aleph_1-\comp}(\gb_\lambda)$.

\noindent
3A) Above $\bar{\gb}$ is smooth \when \, $\theta \in \gb_\lambda
\Rightarrow \gb_\theta \subseteq \gb_\lambda$.

\noindent
4) We say above is $\bar{\gb}$ full when $\gc = \ga \cap$ 
{\rm ps-pcf}$_{\aleph_1-\comp}(\ga)$.
\end{definition}

\begin{remark}
\label{e11}
1) Note that \ref{r19} gives sufficient conditions for the existence
of $\bar{\gb}$ as in \ref{e10}(1).

\noindent
2) Of course, Definition \ref{e10} is interesting particularly when
   $\ga =$ {\rm ps-pcf}$_{\aleph_1-\text{com}}(\ga)$.
\end{remark}

\begin{PROOF}{\ref{e7}}
\underline{Proof of \ref{e7}}
\mn
\begin{enumerate}
\item[$(*)_1$]   $\gc =$ {\rm ps-pcf}$_{\aleph_1-\text{com}}(\gc)$.
\end{enumerate}
\mn
[Why?  By \ref{e5} using AC$_{\cP(\ga)}$.]
\mn
\begin{enumerate}
\item[$(*)_2$]   there is a generating sequence $\langle
\gb_\lambda:\lambda \in \gc\rangle$ for $\ga$.
\end{enumerate}
\mn
[Why?  By \ref{r19}(3) using also AC$_{\gc}$.]
\mn
\begin{enumerate}
\item[$(*)_3$]   let $\gb^*_\lambda = 
\text{ ps-pcf}_{\aleph_1-\text{com}}(\gb_\lambda)$ for $\lambda \in \gc$.
\sn
\item[$(*)_4$]   $(a) \quad \bar\gb^* = 
\langle \gb^*_\lambda:\lambda \in \gc\rangle$ is well defined
\sn
\item[${{}}$]   $(b) \quad \gb^*_\lambda \subseteq \gc$
\sn
\item[${{}}$]   $(c) \quad \gb^*_\lambda = \text{
ps-pcf}_{\aleph_1-\text{com}}(\gb^*_\lambda)$
\sn
\item[${{}}$]   $(d) \quad \lambda = \text{ max}(\gb^*_\lambda)$. 
\end{enumerate}
\mn
[Why?  First, $\bar{\gb}^*$ is well defined as $\bar{\gb} = \langle
\gb_\lambda:\lambda \in \gc\rangle$ is well defined.  Second,
$\gb^*_\lambda \subseteq \gc$ as $\gb_\lambda \subseteq \ga$ hence
$\gb^*_\lambda = \text{ ps-pcf}_{\aleph_1-\text{com}}(\gb^*_\lambda)
\subseteq \text{ ps-pcf}_{\aleph_1-\text{com}}(\ga) = \gc$.  Third,
$\gb^*_\lambda = \text{ ps-pcf}_{\aleph_1-\text{com}}(\gb^*_\lambda)$
by Conclusion \ref{e5}, it is easy to check that its assumption holds
recalling $\gb_\lambda \subseteq \ga$.  Fourth, $\lambda \in
\gb^*_\lambda$ as $J_{=\lambda}[\ga]$ witness $\gb^*_\lambda =
 \text{ ps-pcf}_{\aleph_1-\text{com}}(\gb_\lambda) \subseteq
\lambda^+$ as $\gb_\lambda \in J_{< \lambda^+}[\ga]$ and lastly,
max$(\gb^*_\lambda) = \lambda$ by $(*)_2$.]

We can now choose $\bar{\cF}$ such that
\mn
\begin{enumerate}
\item[$(*)_5$]   $(a) \quad \bar{\cF} = \langle \bar{\cF}_\lambda:\lambda
\in \gc \rangle$
\sn
\item[${{}}$]   $(b) \quad \bar{\cF}_\lambda = \langle
\cF_{\lambda,\alpha}:\alpha < \lambda\rangle$
\sn
\item[${{}}$]   $(c) \quad \bar{\cF}_\lambda$ witness $\lambda =$ 
{\rm ps-tcf}$(\Pi \ga,<_{J_{=\lambda}[\ga]})$
\sn
\item[${{}}$]   $(d) \quad$ if $\lambda \in \ga,\alpha < \lambda$ and
$f \in \cF_{\lambda,\alpha}$ then $f(\lambda) = \alpha$.
\end{enumerate}
\mn
[Why?  For each $\lambda$ there is such $\bar{\cF}$ as $\lambda =$ 
{\rm ps-tcf}$(\Pi \ga,<_{J_{=\lambda}[\ga]})$.  
But we are assuming AC$_{\gc}$ and for clause (d) it is easy; in fact
it is enough to use AC$_{\cP(\ga)}$ and $h$ as in \ref{e5}, getting
$\langle \bar{\cF}_{\gb}:\gb \in \cP(\ga)\rangle,\bar{\cF}_{\gb}$
witness $h(\gb) = \text{ ps-tcf}(\Pi \ga,<_{J_{=\lambda}}[\ga])$ and
putting $\langle \bar{\cF}_{\gb}:\gb \in h^{-1}\{\lambda\}\rangle$
together for each $\lambda \in \gc$.]
\mn
\begin{enumerate}
\item[$(*)_6$]   $(a) \quad$ for $\lambda \in \gc$ and $f \in \Pi
\gb_\lambda$ let $f^{[\lambda]} \in \Pi \gb^*_\lambda$ be defined by:
$f^{[\lambda]}(\theta) =$

\hskip25pt $\text{min}\{\alpha < \lambda$: for every $g \in
\cF_{\theta,\alpha}$ we have 

\hskip25pt $f \rest \gb_\lambda \le (g \rest
\gb_\lambda)$ mod $J_{=\theta}[\gb_\lambda]\}$
\sn
\item[${{}}$]   $(b) \quad$ for $\lambda \in \gc$ and $\alpha <
\lambda$ let $\cF^{[*]}_{\lambda,\alpha} = \{f^{[\lambda]}:f \in
\cF_{\lambda,\alpha}\}$.
\end{enumerate}
\mn
Now
\mn
\begin{enumerate}
\item[$(*)_7$]   $(a) \quad f^{[\lambda]} \rest \ga \ge f$ for $f \in
\Pi \gb_\lambda,\lambda \in \gc$
\sn
\item[${{}}$]   $(b) \quad \langle \cF^*_{\lambda,\alpha}:\lambda \in
\gc,\alpha < \lambda \rangle$ is well defined (hence exist)
\sn
\item[${{}}$]   $(c) \quad \cF^*_{\lambda,\alpha} \subseteq \Pi \gb^*_\lambda$.
\end{enumerate}
\mn
[Why?  Obvious.]
\mn
\begin{enumerate}
\item[$(*)_8$]   let $J_\lambda$ be the $\aleph_1$-complete ideal on
$\gb^*_\lambda$ generated by $\{\gb^*_\theta \cap \gb^*_\lambda:\theta \in
\gc \cap \lambda\}$
\sn
\item[$(*)_9$]   $J_\lambda \subseteq J^{\aleph_1-\comp}_{<
\lambda} [\gb^*_\lambda]$.
\end{enumerate}
\mn
[Why?  As for $\theta_0,\dotsc,\theta_n \ldots \in \gc \cap \lambda$ we
have ps-pcf$_{\aleph_1-\comp}(\cup\{\gb^*_{\theta_n}:n <
\omega\}) = \cup\{\gb^*_{\theta_n}:n < \omega\} \in 
J^{\aleph_1-\comp}_{<\lambda}$.]
\mn
\begin{enumerate}
\item[$\odot_1$]   if $\lambda \in \gc$ and $\alpha_1 < \alpha_2 <
\lambda$ and $f_\ell \in \cF_{\lambda,\alpha_\ell}$ for $\ell=1,2$ \then
\, $f^{[\lambda]}_1 \le f^{[\lambda]}_2$ mod $J_\lambda$.
\end{enumerate}
\mn
[Why?  Let $\ga_* = \{\theta \in \gb_\lambda:f_1(\sigma) \ge
f_2(\theta)\}$, hence by the assumption on $\langle
\cF_{\lambda,\alpha}:\alpha < \lambda\rangle$ we have $\ga_* 
\in J^{\aleph_1-\comp}_{< \lambda}[\ga]$, hence 
we can find a seqauence $\langle \theta_n:n <
\bold n \le \omega \rangle$ such that $\theta_n \in \gc \cap \lambda$ and
$\ga_* \subseteq \gb_* := \cup\{\gb_{\theta_n}:n < \bold n\}$
hence $\gc_* =$ {\rm ps-pcf}$_{\aleph_1-\text{com}}(\ga_*) \subseteq
\cup\{\gb^*_{\theta_n}:n < \bold n\} \in J_\lambda$.  So it suffices to
prove $f^{[\lambda]}_1 \rest (\gb^*_\lambda \backslash \gc_*) \le
f^{[\lambda]}_2 \rest (\gb^*_\lambda \backslash \gc_*)$, so let
$\theta \in \gb^*_\lambda \backslash \bigcup\limits_{n}
\gb^*_{\theta_n}$ clearly $\theta \le \lambda$, let $\alpha :=
f^{[\lambda]}_2(\theta)$, so $(\forall g \in \cF_{\theta,\alpha})(f_2
\rest \gb_\lambda) \le (g \rest \gb_\lambda)$ mod
$J_{=\theta}[\gb_\lambda])$ but $\ga_* \in 
J^{\aleph_1-\text{comp}}_{=\theta}[\gb_\lambda]$ so
$(f_1 \rest \gb_*) \le (f_2 \rest \gb_*) \le (g \rest \gb_*)$ mod
$J_{=\theta}[\gb_*]$ hence $f^{[\lambda]}_1(\theta) \le \alpha =
f^{[\lambda]}_2(\theta)$.  So we are done.]
\mn
\begin{enumerate}
\item[$\odot_2$]   if $\lambda \in \gc$ and $g \in \Pi \gb^*_\lambda$ 
\then \, for some $\alpha < \lambda$ and $f \in \cF_{\lambda,\alpha}$ we
have $g < f$ mod $J_\lambda$.
\end{enumerate}
\mn
[Why?  We choose $\langle h_\theta:\theta \in \gb^*_\lambda\rangle$
such that $h_\theta \in \cF_{\theta,g(\theta)}$ for each $\theta \in
\gb^*_\lambda$.  Let $h_1 \in \Pi \gb^*_\lambda$ be defined by $h_1(\kappa) =
\sup\{h^{[\lambda]}_\theta(\kappa):\kappa \in \gb_\theta$ and $\theta
\in \gb^*_\lambda\}$ hence there are $\alpha < \lambda$ and $h_2 \in
\cF_{\lambda,\alpha}$ such that $h_1 \le h_2$ mod $J_{=\lambda}[\ga]$.
Now $f := h_2^{[\lambda]} \in \Pi \gb^*_\lambda$ is as required.] 
\mn
\begin{enumerate}
\item[$\odot_3$]    $\cF_{\lambda,\alpha}$ witness $\lambda =$
{\rm ps-tcf}$(\Pi \gb^*_\lambda,<_{J_\lambda})$.
\end{enumerate}
\mn
[Why?  In $(*)_7 + \odot_1 + \odot_2$.]

So
\mn
\begin{enumerate}
\item[$\odot_4$]    $\bar{\gb}^* = \langle \gb^*_\lambda:\lambda \in
\gc \rangle$ is a generating sequence for $\gc$.
\end{enumerate}
\mn
[Why?  Check.]
\end{PROOF}

\begin{remark}
\label{e30}
Clearly $\bar\gb^*$ is closed, but what about smoooth?  Is this
necessary for generalizing \cite{Sh:460}?
\end{remark}

\begin{discussion}
\label{e31}
Naturally the definition now of $\bar{\cF}$ as in \ref{e10}(2) for 
$\Pi{\ga}$ is more involved where $\bar{\cF} = \langle
\bar{\cF}_\lambda:\lambda \in$ {\rm ps-pcf}$_{\kappa-\text{com}}
(\ga)\rangle,\bar{\cF}_\lambda = \langle
\cF_{\lambda,\alpha}:\alpha < \lambda\rangle$ exemplifies
ps-tcf$(\Pi{\ga},J_{=\lambda}(\ga))$. 
\end{discussion}

\begin{claim}
\label{e33}
Assume
\mn
\begin{enumerate}
\item[$(a)$]  $\ga$ a set of regular cardinals
\sn
\item[$(b)$]   $\kappa$ is regular $> \aleph_0$
\sn
\item[$(c)$]   $\gc = \text{\rm ps-pcf}_{\kappa-\text{\rm com}}(\ga)$
\sn
\item[$(d)$]  {\rm min}$(\ga)$ is $\ge \hrtg(\cP(\gc))$ or at least $\ge
\hrtg(\gc)$
\sn
\item[$(e)$]  $\bar{\cF} = \langle \bar{\cF}_\lambda:\lambda \in
\gc\rangle,\bar{\cF}_\lambda = \langle \cF_{\lambda,\alpha}:\alpha <
\lambda\rangle$ witness $\lambda =$ {\rm ps-tcf}$(\Pi \ga,
<^{J^{\kappa-\comp}}_{=\lambda}[\ga])$.
\end{enumerate}
\mn
Then
\mn
\begin{enumerate}
\item[$\boxplus$]  for every $f \in \Pi \ga$ for some $g \in \Pi \gc$,
if $g \le g_1 \in \Pi \gc$ and $\bar h \in
\Pi\{\cF_{\lambda,g_1(\lambda)}:\lambda \in \gc\}$ 
\then \, $(\exists \gd \in [\gc]^{< \kappa}) (f < 
\sup\{h_\lambda:\lambda \in \gd\})$.
\end{enumerate}
\end{claim}

\begin{PROOF}{\ref{e33}}
Let $f \in \Pi \ga$.  For each $\lambda \in \text{
ps-pcf}_{\kappa-\text{com}}(\ga)$ let $\alpha_{f,\lambda} = \text{
min}\{\alpha < \lambda:f < g$ mod $J_{= \lambda}[\ga]$ for every $g
\in \cF_{\lambda,\alpha}\}$ so clearly each $\alpha_f$ is well defined
hence $\bar \alpha = \langle \alpha_{f,\lambda}:\lambda \in$ 
{\rm ps-pcf}$_{\kappa-\text{com}}(\ga)\rangle$ exists.   Let $\langle
g_\lambda:\lambda \in \gc \rangle$ be any sequence from $\Pi
\cF_{\lambda,\alpha_{f,\lambda}}$ at least one exists when AC$_{\gc}$.
Let $\ga_{f,\lambda} = \{\theta \in \ga:f(\theta) <
g_\lambda(\theta)\}$ so $\langle \ga_{f,\lambda}:\lambda \in
\gc\rangle$ exists and we claim that for some $\gd \in [\gc]^{<
\kappa}$ we have $\ga = \cup\{a_{f,\lambda}:\lambda \in \gd\}$. 
Otherwise let $J$ be the $\kappa$-complete ideal on $\ga$
generated by $\{\ga_{f,\lambda}:\lambda \in \gc\}$, it is a
$\kappa$-complete ideal.  So the ``no-hole claim", \ref{r22} applicable by our
assumptions there is a $\kappa$-complete ideal $J$ on $\ga$ extending
$J$ such that $\lambda_* = \text{ ps-tcf}(\Pi \ga,<_{J_1})$ is well
defined.  So $\lambda_* \in \gc$ and $\ga_{f,\lambda_*} \in J_1$, easy
contradiction. 
\end{PROOF}
\newpage

\section {Measuring reduced products}

\subsection {On ps-$\bold T_D(g)$} \

Now we consider some ways to measure the size of ${}^\kappa \mu/D$ and
show that they essentially are equal; see Discussion \ref{r38}.

\begin{definition}
\label{r26}  
Let $\bar\alpha = \langle \alpha_y:y \in Y\rangle \in {}^Y\Ord$ be 
such that $t \in Y \Rightarrow \alpha_t > 0$.

\noindent
1) For $D$ a filter on $Y$ let ps-$\bold T_D(\bar\alpha) 
= \sup\{\hrtg(\bold F):\bold F$ is a family 
of non-empty subsets of $\Pi\bar\alpha$ such that
for every ${\cF}_1 \ne {\cF}_2$ from $\bold F$ we have $f_1 \in
\cF_1 \wedge f_2 \in {\cF}_2 \Rightarrow f_1 \ne_D f_2\}$, recalling
$f_1 \ne_D f_2$ means $\{y \in Y:f_1(y) \ne f_2(y)\} \in D$. 

\noindent
2) Let ps-$\bold T_{\kappa-\comp}(\bar\alpha) =
\sup\{\hrtg(\bold F)$: for some $\kappa$-complete filter $D$ on
$Y$, $\bold F$ is as above for $D\}$.

\noindent
3) If we allow $\alpha_t = 0$ just replace $\Pi \bar\alpha$ by $\Pi^*
\bar\alpha := \{f:f \in \prod\limits_{t}(\alpha_t +1)$ and
   $\{t:f(t) = \alpha_t\} = \emptyset$ mod $D\}$.
\end{definition}

\begin{theorem}
\label{r29}
[{\rm DC + AC}$_{\cP(Y)}$]  Assume that $D$ is a $\kappa$-complete 
filter on $Y$ and $\kappa > \aleph_0$ and $g \in {}^Y$({\rm Ord}  
$\backslash \{0\})$, if $g$ is constantly $\alpha$ we may write
$\alpha$.  The following cardinals are equal \underline{or} at 
least $\lambda_1,\lambda_2,\lambda_3$ are
$\text{\rm Fil}^1_\kappa(D)$-almost equal which means: for $\ell_1,\ell_2 \in
\{1,2,3\}$ we have $\lambda_{\ell_1} \le^{\text{\rm sal}}_S
\lambda_{\ell_2}$ which means if $\alpha < \lambda_{\ell_1}$ then $\alpha$ is
included in the union of $S$ sets each of
order type $< \lambda_{\ell_2}$:
\mn
\begin{enumerate}
\item[$(a)$]  $\lambda_1 = \sup\{|\rk_{D_1}(g)|^+:D_1 \in$ 
{\rm Fil}$^1_\kappa(D)\}$
\sn
\item[$(b)$]  $\lambda_2 = \sup\{\lambda^+$: there are $D_1 \in$ 
{\rm Fil}$^1_\kappa(D)$ and a $<_{D_1}$-increasing sequence $\langle
\cF_\alpha:\alpha < \lambda\rangle$ such that $\cF_\alpha \subseteq
\prod\limits_{y \in Y} g(y)$ is non-empty$\}$
\sn
\item[$(c)$]  $\lambda_3 = \sup\{\text{\rm ps}-\bold T_{D_1}(g):
D_1 \in \text{\rm Fil}^1_\kappa(D)\}$.
\end{enumerate}
\end{theorem}

\begin{remark}
\label{r30}
1) Recall that for $D$ a $\kappa$-complete filter on $Y$ we let
   Fil$^1_\kappa(D) = \{E:E$ is a $\kappa$-complete filter on $Y$
 extending $D\}$.

\noindent
2) The conclusion gives slightly less than equality of
$\lambda_1,\lambda_1,\lambda_3$.

\noindent
3) See \ref{k1}(6) below, by it $\lambda_2 =$ 
{\rm ps-Depth}$^+({}^\kappa \mu,<_D)$.

\noindent
4) We may replace $\kappa$-complete by $(\le Z)$-complete if
$\aleph_0 \le |Z|$.

\noindent
5) Compare with Definition \ref{k1}.

\noindent
6) Note that those cardinals are $\le \hrtg(\Pi^* g)$.
\end{remark}

\begin{PROOF}{\ref{r29}}
\bigskip

\noindent
\underline{Stage A}:  $\lambda_1 \le^{\text{sal}}_{\text{Fil}^1_\kappa(D)}
 \lambda_2,\lambda_3$.

Why?  Let $\chi < \lambda_1$, so by clause (a) 
there is $D_1 \in \text{ Fil}^1_\kappa(D)$
such that rk$_{D_1}(g) \ge \chi$.
Let $X_{D_2} = \{\alpha < \chi$: some $f \in \prod\limits_{y \in Y}
g(y)$ satisfies $D_2 = \text{ dual}(J[f,D_1])$ and $\alpha = 
\rk_{D_1}(f)\}$, for any
$D_2 \in \text{ Fil}^1_\kappa(D_1)$.  By \cite[1.11(5)]{Sh:938} 
 we have $\chi = \bigcup\{X_{D_2}:D_2 \in$ {\rm Fil}$^1_\kappa(D_1)\}$. 

Now
\mn
\begin{enumerate}
\item[$\odot$]  $D_2 \in$ {\rm Fil}$^1_\kappa(D_1) \Rightarrow 
|\otp(X_{D_2})| < \lambda_2,\lambda_3$; this is enough.
\end{enumerate}
\mn
Why?  Letting $\cF_{D_2,i}= \{f \in {}^Y \mu:
\text{ rk}_{D_1}(f) = i$ and $J[f,D_1] = \text{ dual}(D_2)\}$, by
\cite[1.11(2)]{Sh:938} we have: $i < j \wedge i 
\in X_{D_2} \wedge j \in X_{D_2} \wedge
f \in \cF_{D_2,i} \wedge g \in \cF_{D_2,j} \Rightarrow f < g$ mod
$D_2$ so otp$(X_{D_2}) < \lambda_2,\lambda_3$.
\medskip

\noindent
\underline{Stage B}:  $\lambda_2 \le^{\text{sal}}_{\text{Fil}^1_\kappa(D)}
\lambda_1,\lambda_3$.

Why?  Let $\chi < \lambda_2$ and let 
$D_1$ and $\langle \cF_\alpha:\alpha < \chi\rangle$
exemplify $\chi < \lambda_2$.  
Let $\gamma_\alpha = \text{ min}\{\text{rk}_{D_1}(f):f
\in \cF_\alpha\}$ so easily $\alpha < \beta < \chi \Rightarrow \gamma_\alpha <
\gamma_\beta$ hence rk$_D(g) \ge \chi$.  So $\chi < \lambda_1$ and
as for $\chi < \lambda_3$ this holds by Definition \ref{r26}(2) as $\alpha
< \beta \wedge f \in \cF_\alpha \wedge g \in \cF_\beta \Rightarrow f <
g$ mod $D_1 \Rightarrow f \ne g$ mod $D_1$ as $\chi^+ = \hrtg(\chi)
\le \lambda_3$.
\medskip

\noindent
\underline{Stage C}:  $\lambda_3 \le^{\text{sal}}_{\text{Fil}^1_\kappa(D)} 
\lambda_1,\lambda_2$.

Why?  Let $\chi < \lambda_3$.
Let $\langle \cF_\alpha:\alpha < \chi\rangle$ exemplify $\chi <
\lambda_3$.  For each $\alpha < \chi$ let $\bold D_\alpha = 
\{\text{dual}(J[f,D]):f
\in \cF_\alpha\}$ so a non-empty subset of $\text{Fil}^1_\kappa(Y)$.
Now for every  $D_1 \in \cup\{\bold D_\alpha:\alpha < \lambda\}$ let
$X_{D_1} = \{\alpha < \chi:D_1 \in \bold D_\alpha\}$ and for $\alpha
\in X_{D_1}$ let $\zeta_{D_1,\alpha} = 
\text{ min}\{\text{rk}_D(f):f \in \cF_\alpha$ and $D_1 
= \text{ dual}(J[f,D])\}$ and let $\cF_{D_1,\alpha} = \{f \in
\cF_\alpha:D_1 = J[f,D]$ and rk$_D(f) = \zeta_{D_1,\alpha}\}$ so a
non-empty subset of $\cF_\alpha$ and clearly $\langle
(\zeta_{D_1,\alpha},\cF_{D_1,\alpha}):\alpha \in X_{D_1}\rangle$
exists.

Now
\mn
\begin{enumerate}
\item[$(a)$]   $\alpha \mapsto \zeta_{D_1,\alpha}$ is a
one-to-one function with domain $X_{D_1}$
\sn
\item[$(b)$]  $\chi = \cup\{X_{D_1}:D_1 \in \text{
Fil}^1_\kappa(Y)\}$
\sn
\item[$(c)$]   if $\alpha < \beta$ are from $X_{D_1}$ and
$\zeta_{D_1,\alpha} < \zeta_{D_1,\beta},f \in \cF_{D_1,\alpha},
g \in \cF_{D_2,\beta}$ then $f < g$ mod $D_1$.
\end{enumerate}
\mn
[Why?  For clause (a), if $\alpha \ne \beta \in X_{\zeta_1},f \in
\cF_{D_1,\alpha},g \in \cF_{D_1,\beta}$ then $f \ne g$ mod $D$ hence by
\cite[1.11]{Sh:938} we have $\zeta_{D_1,\alpha} \ne
\zeta_{D_1,\beta}$.  For clause (b), it follows by the choices of
$\bold D_\alpha,X_{D_1}$.  Lastly, clause (c) follows by
\cite[1.11(2)]{Sh:938}.] 

Hence (by clause (c))
\mn
\begin{enumerate}
\item[$(d)$]  otp$(X_{D_1})$ is $< \lambda_2$ and is $\le \text{
rk}_{D_1}(g)$ for $D_1 \in \cup\{\bold D_\alpha:\alpha < \chi\}
\subseteq \text{ Fil}^1_\kappa(D)$.
\end{enumerate}
\mn
Together clauses (b),(d) show that $\chi < \lambda_1,\lambda_2$ so we
are done. 
\end{PROOF}

\begin{observation}
\label{r31}
Assume $D$ is a filter on $Y$ and $\bar\alpha \in {}^Y(\Ord \backslash
\{0\})$.

\noindent
1) $\ps-\bold T_D(\bar\alpha) = \sup\{\lambda^+$: there is a sequence
   $\langle \cF_\alpha:\alpha < \lambda\rangle$ such that $\cF_\alpha
   \subseteq \Pi \bar\alpha,\cF_\alpha \ne \emptyset$ and $\alpha \ne
   \beta \wedge f_1 \in \cF_\alpha \wedge f_2 \in \cF_\beta
   \Rightarrow f_1 \ne_D f_2\}$.
\end{observation}

\begin{proof}
1) Clearly the new definition gives a cardinal $\le \ps-\bold
   T_D(\bar\alpha)$.  For the other inequality assume $\lambda <
   \ps-\bold T_D(\bar\alpha)$ so there is $\bold F$ as there such that
   $\lambda < \hrtg(\bold F)$.  As $\lambda < \hrtg(\bold F)$ there is
   a function $h$ from $\bold F$ onto $\lambda$.  For $\alpha <
   \lambda$ define $\cF'_\alpha = \cup\{\cF:\cF \in \bold F$ and
   $h(\cF) = \alpha\}$.  So $\langle \cF'_\alpha:\alpha <
   \lambda\rangle$ exists and is as required.
\end{proof}
\medskip

\noindent
Concerning Theorem \ref{r29} we may wonder ``when does
$\lambda_1,\lambda_2$ being $S$-almost equal implies they are equal".
\begin{definition}
\label{r32}
1) We say ``the power of $\cU_1$ is $S$-almost smaller than the power of
   $\cU_2$", or write $|\cU_1| \le |\cU_2|$ mod $S$ or $|\cU_1|
   \le^{\alm}_S |\cU_2|$ when: we can find a sequence $\langle
 u_{1,s}:s \in S\rangle$ such that $\cU_1 = \cup\{u_{1,s}:s \in S\}$ and $s
   \in S \Rightarrow |\cU_{1,s}| \le |\cU_2|$.

\noindent
2) We say the power $|\cU_1|,|\cU_2|$ are $S$-almost equal (or
$|\cU_1| = |\cU_2|$ mod $S$ or $|\cU_1| =^{\alm}_S |\cU_2|$)
\when \, $|\cU_1| \le^{\alm}_S |\cU_2| \le^{\alm}_S |\cU_2|$.

\noindent
3) Let $|\cU_1| \le^{\alm}_{<S} |\cU_2|$ be defined naturally.

\noindent
4) In particular this applies to cardinals.

\noindent
5) Let $|\cU_1| <^{\alm}_S |\cU_2|$ means there is a sequence
 $\langle u_{1,s}:s \in S\rangle$ with union $\cU_1$ such that $s \in S
\Rightarrow |\cU_s| < |\cU_2|$.

\noindent
6) Let $|\cU_1| \le^{\sal}_S |\cU_2|$ means that if
$|\cU| < |\cU_1|$ then $|\cU| <^{\alm}_S |\cU_2|$.
\end{definition}

\begin{observation}
\label{r34}
1) If $|\cU_1| \le |\cU_2|$ and $S \ne \emptyset$ \then \, $|\cU_1|
\le^{\alm}_S |\cU_2|$.

\noindent
2) If $\lambda_1 \le \lambda_2$ and $S \ne \emptyset$ then $\lambda_1 
\le^{\sal}_S \lambda_2$.
\end{observation}

\begin{PROOF}{\ref{r34}}
Immediate.
\end{PROOF}

\begin{observation}
\label{r36}
1) The cardinals $\lambda_1,\lambda_2$ are equal \when \, $\lambda_1 
=^{\alm}_S \lambda_2$ and cf$(\lambda_1)$, cf$(\lambda_2) \ge
\hrtg(\cP(S))$.

\noindent
2) The cardinals $\lambda_1,\lambda_2$ are equal \when \, $\lambda_1
   =^{\alm}_S \lambda_2$ and $\lambda_1,\lambda_2$ are limit
cardinals $> \hrtg(\cP(S))$.

\noindent
3) If $\lambda_1 \le^{\alm}_S \lambda_2$ and $\partial =
\hrtg(\cP(S))$ \then \, $\lambda_1 \le^{\alm}_{< \partial} \lambda_2$.

\noindent
4) If $\lambda_1 \le^{\alm}_{< \theta} \lambda_2$ and
   cf$(\lambda_1) \ge \theta$ \then \, $\lambda_1 \le \lambda_2$.

\noindent
5) If $\lambda_1 \le^{\alm}_{< \theta} \lambda_2$ and
$\theta \le \lambda^+_2$ \then \, $\lambda_1 \le \lambda^+_2$.
\end{observation}

\begin{PROOF}{\ref{r36}}
1) Otherwise, let $\partial = \hrtg(\cP(S))$, 
 \wilog \, $\lambda_2 < \lambda_1$ and by part (3) we have
$\lambda_1 \le^{\alm}_{< \partial} \lambda_2$ and by part (4) we have
   $\lambda_1 \le \lambda_2$ contradiction.

\noindent
2) Otherwise letting $\partial = \hrtg(\cP(S))$ 
 \wilog \, $\lambda_2 < \lambda_1$ and by part (3) we have
$\lambda_1 \le^{\alm}_{< \partial} \lambda_2$ but $\theta <
   \lambda_2$ is assume and $\lambda^+_2 < \lambda_1$ as $\lambda_2$
   is a limit cardinal so together we get contradiction to part (5).

\noindent
3) If $\langle u_s:S \in S\rangle$ witness $\lambda_1
   \le^{\alm}_S \lambda_2$, let $w = \{\alpha < \lambda_1$: for
   no $\beta < \alpha$ do we have $(\forall s \in S)(\alpha \in u_s
   \equiv \beta \in u_s)\}$ so clearly $|w| < \hrtg(\cP(S)) = \theta$ and for 
$\alpha \in w$ let 
$v_\alpha = \{\beta < \lambda_1:(\forall s \in S)(\alpha \in u_s
   \equiv \beta \in u_s)\}$ so $\langle v_\alpha:\alpha \in w \rangle$
witness $\lambda_1 \le^{\text{alm}}_w \lambda_2$ hence $\lambda_1
\le^{\alm}_{< \theta} \lambda_2$.

\noindent
4),5)  Let $\sigma <  \theta$ be such that $\lambda_1
   \le^{\alm}_\sigma \lambda_2$ and let $\langle
   u_\varepsilon:\varepsilon < \sigma\rangle$ witness $\lambda_1
   \le^{\alm}_\sigma \lambda_2$, that is $|u_\varepsilon| \le
   \lambda_2$ for $\varepsilon < \sigma$ and
   $\cup\{u_\varepsilon:\varepsilon < \sigma\} = \lambda_1$.  

For part (4), if $\lambda_2 < \lambda_1$, then we have $\varepsilon < \sigma
   \Rightarrow |u_\varepsilon| < \lambda_1$, but cf$(\lambda_1) >
   \sigma$ hence $|\{\cup\{u_\varepsilon:\varepsilon < \sigma\} <
   \lambda_1|$, contradiction.  

For part (5) for $\varepsilon < \sigma$, let $u'_\varepsilon =
   u_\varepsilon \backslash \cup\{u_\zeta:\zeta < \varepsilon\}$ and
   so otp$(u'_\varepsilon) \le \otp(u_\varepsilon) <
   |u_\varepsilon|^+ \le \lambda^+_2$ so easily $|\lambda_1| =
|\cup\{u_\varepsilon:\varepsilon < \sigma\}| =
   |\cup\{u'_\varepsilon:\varepsilon < \sigma\}| \le \sigma \cdot
\lambda^+_2 \le \lambda^+_2 \cdot \lambda^+_2 = \lambda^+_2$. 
\end{PROOF}

Similarly
\begin{observation}
\label{r37}
1) If $\lambda_1 <^{\text{alm}}_S \lambda_2$ and $\partial =
\hrtg(\cP(S))$ \then \, $\lambda_1 <^{\alm}_{< \partial}
\lambda_2$.

\noindent
2) If $\lambda_1 <^{\alm}_{< \theta} \lambda_2$ and
cf$(\lambda_1) \ge \theta$ \then \, $\lambda_1 < \lambda_2$.

\noindent
3) If $\lambda_1 <^{\alm}_{< \theta} \lambda_2$ and $\theta \le
\lambda^+_2$ \then \, $\lambda_1 \le \lambda_2$.

\noindent
4) If $\lambda_1 \le^{\sal}_S \lambda_2$ and $\partial =
\hrtg(\cP(S))$ \then \, $\lambda_1 <^{\sal}_{< \partial}
\lambda_2$.

\noindent
5) If $\lambda_1 \le^{\sal}_{< \theta} \lambda_2$ and
$\partial \le \lambda^+_2,\theta < \lambda_2$ and cf$(\lambda_2)
\ge \theta$ \then \, $\lambda_1 \le \lambda_2$.

\noindent
6) If $\lambda_1 \le^{\sal}_{< \theta} \lambda_2$ and
$\theta \le \lambda^+_2$ \then \, $\lambda_1 \le \lambda^+_2$.
\end{observation}

\begin{discussion}
\label{r38}  
1) We like to measure $({}^Y \mu)/D$ in some ways and show
their equivalence, as was done in ZFC.  Natural candidates are:
\medskip

\noindent
\begin{enumerate}
\item[$(A)$]   pp$_D(\mu)$: say of length of increasing sequence
$\bar P$ (not $\bar p$!, i.e. sets) ordered by $<_D$
\smallskip

\noindent
\item[$(B)$]  {\rm pp}$^+_Y(\mu) =
\sup\{\text{pp}^+_D(\mu):D$ an $\aleph_1$-complete filter on $Y\}$
\smallskip

\noindent
\item[$(C)$]  As in  \ref{r26}.
\end{enumerate}
\mn
2) We may measure ${}^Y \mu$ by considering all $\partial$-complete
filters.

\noindent
3) We may be more lenient in defining ``same cardinality".  E.g.
\mn
\begin{enumerate}
\item[$(A)$]  we define when sets have similar powers say by divisions
to ${\cP}({\cP}(Y))$ sets we measure $({}^Y \mu)/\approx_{\cP(\cP(Y))}$
where $\approx_B$ is the following equivalence relation on sets:

$X \approx_B Y$ when we can find sequences $\langle X_b:b \in B
\rangle,\langle Y_b:b \in B\rangle$ 

such that:
\begin{enumerate}
\item[$(a)$]  $X = \cup\{X_b:b \in B\}$
\sn
\item[$(b)$]  $Y = \cup\{Y_b:b \in B\}$
\sn
\item[$(c)$]  $|X_b| = |Y_b|$
\end{enumerate}
\item[$(B)$]   we may demand more: the $\langle X_b:b \in B\rangle$
are pairwise disjoint and the $\langle Y_b:b \in B\rangle$ are
pairwise disjoint
\sn
\item[$(C)$]  we may demand less: e.g.
\mn
\begin{enumerate}
\item[$(c)'$]  $|X_b| \le_* |Y_b| \le_* |X_b|$

and/or
\item[$(c)_*$]  $(\forall b \in B)(\exists c \in B)(|X_b| \le |Y_c|)$
and

$(\forall b \in B)(\exists c \in B)(|Y_b| \le |X_c|)$.
\end{enumerate}
\end{enumerate}
\mn
Note that some of the main results of \cite{Sh:835} can be expressed
this way.
\mn
\begin{enumerate}
\item[$(D)$]  rk-sup$_{Y,\partial}(\mu) =$ {\rm rk-sup}
$\{\rk_D(\mu):D$ is $\partial$-complete filters on $Y\}$
\sn
\item[$(E)$]  for each non-empty $X \subseteq {}^Y \mu$ let

\[
\text{sp}^1_\alpha(X) = \{(D,J):D \text{ an } \aleph_1\text{-complete
filter on } Y,J = J[f,D],\alpha = \text{ rk}_D(f) \text{ and } f \in X\}
\]

\[
\text{sp}_1(X) = \cup\{\text{sp}^1_\alpha(X):\alpha\}
\] 
\smallskip

\noindent
\item[$(F)$]   \und{question}:  If $\{\text{sp}(X_s):s \in S\}$ is constant,
can we bound $J$?
\smallskip

\noindent
\item[$(G)$]  $X,Y$ are called connected \underline{when} 
sp$(X_1),\text{sp}(X_2))$ are non-disjoint or equal.
\end{enumerate}
\mn
4) We hope to prove, at least sometimes $\gamma := 
\Upsilon({}^Y \mu) \le \text{\rm pp}_\kappa(\mu)$
that is we like to immitate \cite{Sh:835} without the choice axioms
on ${}^\omega \mu$.
\mn
So there is $\bar f = \langle f_\alpha:\alpha < \delta\rangle$
witnessing $\gamma < \Upsilon({}^Y \mu)$.  We define $u = u_{\bar f}
= \{\alpha$: there is no $\bar \beta \in {}^\omega \alpha$ such that
$(\forall i t \in Y)(f_\alpha(t) \in \{f_{\beta_n}(t):n < \omega\})$.
You may say that $u_{\bar f}$ is the set of $\alpha < \delta$ such
that $f_\alpha$ is ``really novel".

By DC this is O.K., i.e.
\mn
\begin{enumerate}
\item[$\boxplus_1$]  for every $\alpha < \delta$ there is $\bar \beta
\in {}^\omega(u_{\bar f} \cap \alpha)$ such that $(\forall t \in
Y)(f_\alpha(t)) = \{f_{\beta_n}(t):n < \omega\}$.
\end{enumerate}
\mn
Next for $\alpha \in u_{\bar f}$ we can define $D_{\bar f,\alpha}$,
the $\aleph_1$-complete filter on $Y$ generated by $\Big\{ \{t \in
Y:f_\beta(t) = f_\alpha(t)\}:\beta < \alpha\Big\}$.  So clearly $\alpha
\ne \beta \in u_{\bar f} \wedge D_{\bar f,\alpha} = D_{\bar f,\beta}
\Rightarrow f_\alpha \ne_D f_\beta$.  Now for each pair $\bar D =
(D_1,D_2) \in \text{ Fil}^4_Y$ (i.e. for the $\aleph_1$-complete case)
let $\Lambda_{\bar f,\bar D} = \{\alpha \in u_{\bar f}:D_{\bar
f,\alpha} = D_1$ and $J[f_\alpha,D_1]\} = \text{ dual}(D_2)$.  So
$\gamma$ is the union of $\le \cP(\cP(Y))$-sets (as $|Y| = |Y| \times
|Y|$, well ordered.

So
\medskip

\noindent
\begin{enumerate}
\item[$(*)_1$]  $\gamma \le \hrtg(({}^Y \omega \times {}^\omega(\mu))$
\smallskip

\noindent
\item[$(*)_2$]  $u$ is the union of ${\cP}({\cP}(\kappa))$-sets each 
of cardinality $< \text{ pp}^+_{Y,\aleph_1}(\mu)$
\smallskip

\noindent
\item[$(I)$]   what about $\hrtg({}^\kappa \mu) < 
\text{ ps-pp}_{Y,\aleph_1}(\mu)$?
\end{enumerate}
We are given $\langle {\cF}_\alpha:\alpha < \kappa\rangle \ne
F_\alpha \ne \emptyset,{\cF}_\alpha \subseteq \mu,\alpha \ne \beta
\Rightarrow {\cF}_\alpha \cap {\cF}_\beta = \emptyset$.

Easier: looking modulo a fix filter $D$.
\medskip

\noindent
\begin{enumerate}
\item[$(*)_2$]   for $D \in \text{ Fil}_{Y,\aleph_1}$, let
${\cF}_{\alpha,D} = \{f \in {\cF}_\alpha:\neg(\exists g \in
{\cF}_\alpha)(g <_D f)\}$.
\end{enumerate}
Maybe we have somewhere a bound on the size of ${\cF}_{\alpha,D}$.
\end{discussion}
\bigskip

\subsection {Depth of Reduced Power of Ordinals} \
\bigskip

Our intention has been to generalize a relative of \cite{Sh:460}, but
actually we are closed to
\cite[\S3]{Sh:513} using IND but unlike \cite{Sh:938} rather than with rank
we deal with depth.

\begin{definition}
\label{k1}
1) Let suc$_X(\alpha)$ be the first ordinal $\beta$ such that we
   cannot find a sequence $\langle \cU_x:x \in X\rangle$ of subsets of
   $\beta$, each of order type $< \alpha$ such that $\beta =
   \cup\{\cU_x:x \in X\}$.

\noindent
2) We define suc$^{[\varepsilon]}_X(\alpha)$ by induction on
   $\varepsilon$ naturally: if $\varepsilon = 0$ it is $\alpha$, if
   $\varepsilon = \zeta +1$ it is
   suc$_X(\text{suc}^{[\zeta]}_X(\alpha))$ and if $\varepsilon$ is a
   limit ordinal then it is $\cup\{\text{suc}^{[\zeta]}_X(\alpha):\zeta
   < \varepsilon\}$.

\noindent
3) For a quasi-order $P$ let the pseudo ordinal depth of $P$, denoted by
ps-o-Depth$(P)$ be sup$\{\gamma$: there is 
a $<_P$-increasing sequence $\langle X_\alpha:\alpha < \gamma\rangle$
of non-empty subsets of $P\}$.

\noindent
4) o-Depth$(P)$ is defined 
similarly demanding $|X_\alpha|=1$ for $\alpha < \gamma$.

\noindent
5) Omitting the ``ordinal" means $\gamma$ is replaced by $|\gamma|$;
   similarly in the other variants.

\noindent
6) Let ps-o-Depth$^+(P) = \sup\{\gamma +1$: there is an increasing
   sequence $\langle X_\alpha:\alpha < \gamma\rangle$ of non-empty
   subsets of $P\}$. 
Similarly for the other variants, e.g. without o we use $|\gamma|^+$
instead of $\gamma +1$ in the supremum. 

\noindent
7) For $D$ a filter on $Y$ and $\bar\alpha \in {}^Y$({\rm Ord}
$\backslash \{0\})$ let ps-o-Depth$^+_D(\bar\alpha) =$
   ps-o-Depth$^+(\Pi \bar\alpha,<_D)$.  Similarly for 
the other variants and we may allow $\alpha_t = 0$ as in \ref{r26}(3).

\noindent
8) Let ps-o-depth$^+_D(\bar\alpha)$ be the cardinality of
   ps-o-Depth$^+_D(\bar\alpha)$.  
\end{definition}

\begin{remark}
Note that \ref{r24} can be phrased using this definition.
\end{remark}

\begin{definition}
\label{k4}
0) We say $\bold x$ is a filter $\omega$-sequence \when \,
$\bold x = \langle (Y_n,D_n):n < \omega\rangle = \langle
Y_{\bold x,n},D_{\bold x,n}:n < \omega\rangle$ is such that 
$D_n$ is a filter on $Y_n$ for each $n < \omega$; 
we may omit $Y_n$ as it is $\cup\{Y:Y
\in D\}$ and may write $D$ if $\bigwedge\limits_{n} D_n=D$.

\noindent
1) Let IND$(\bold x),\bold x$ has the independence property,  
means that for every sequence $\bar F =
\langle F_{m,n}:m < n < \omega\rangle$ from alg$(\bold x)$, see below,
there is $\bar t \in \prod\limits_{n < \omega} Y_n$ such that $m <
n < \omega \Rightarrow t_m \notin F_{m,n}(\bar t \rest (m,n])$.  Let
NIND$(\bold x)$ be the negation.

\noindent
2) Let alg$(\bold x)$ be the set of sequence $\langle F_{n,m}:n < m <
 \omega\rangle$ such that $F_{m,n}:\prod\limits_{\ell=m+1}^{n} Y_\ell
   \rightarrow \text{ dual}(D_n)$.

\noindent
3) We say $\bold x$ is $\kappa$-complete \when \, each $D_{\bold x,n}$
   is a $\kappa$-complete filter.
\end{definition}

\begin{theorem}
\label{k6}
Assume {\rm IND}$(\bold x)$ where $\bold x = \langle (Y_n,D_n):
n < \omega\rangle$ is as in Definition \ref{k4},
$D_n$ is $\kappa_n$-complete, $\kappa_n \ge \aleph_1$.

\noindent
1) [{\rm DC + AC}$_{Y_n}$ for $n < \omega$]  
For every ordinal $\zeta$, for infinitely many $n$'s 
{\rm ps-o-Depth}$({}^{(Y_n)}\zeta,<_{D_n}) \le \zeta$.

\noindent
2) [{\rm DC}]  For every ordinal $\zeta$ for infinitely many $n$, 
{\rm o-Depth}$({}^{(Y_n)}\zeta,<_{D_n}) \le \zeta$,
equivalently there is no $<_{D_n}$-increasing sequence of length
$\zeta +1$.
\end{theorem}

\begin{remark}
\label{k7}  
0) Note that the present results are incomparable with
\cite[\S4]{Sh:938} - the loss is using depth instead of rank and
possibly using ``pseudo".

\noindent
1) [Assume AC$_{\aleph_0}$]
If in \ref{k6}, for every $n$ we have rk$_{D_n}(\zeta) 
> \text{ suc}_{\text{Fil}^1_\kappa(D_n)}(\zeta)$ \then \, for some
$D^1_n \in \text{ Fil}^1_{\aleph_1}(Y_n)$ for 
$n < \omega$ we have NIND$(\langle
Y_n,D^1_n):n < \omega\rangle$.  (Why?  By \cite[5.9]{Sh:938}).  
But we do not know much on the $D^1_n$'s.

\noindent
2) This theorem applies to e.g. $\zeta = \aleph_\omega,Y_n = \aleph_n,D_n =
   \text{ dual}(J^{\text{bd}}_{\aleph_n})$.  So even in ZFC, 
it tells us things not covered by \cite[\S3]{Sh:513}.  Note that Depth
   and pcf are closely connected but only for sequences of length $\ge
\hrtg(\cP(Y))$. 

\noindent
3) If we assume {\rm IND}$(\langle Y_{\eta(n)},D_{\eta(n)}:n < \omega\rangle$
   for every increasing $\eta \in {}^\omega \omega$, which is quite
   reasonable \then \, in Theorem \ref{k6} we can strengthen the
   conclusion, replacing ``for infinitely many $n$'s" by 
``for every $n < \omega$ large enough".

\noindent
4) Note that \ref{k6}(2) is complimentary to \cite{Sh:835}.
\end{remark}

\begin{observation}
\label{k8}
If $\bold x$ is a filter $\omega$-sequence and $n_* < \omega$ and
IND$(\bold x \rest [n_*,\omega)$ \then \, IND$(\bold x)$.
\end{observation}

\begin{PROOF}{\ref{k8}}
Let $\bar F = \langle F_{n,m}:n < m < \omega\rangle \in \text{
alg}(\bold x)$, so $\langle F_{n,m}:n \in [n_*,\omega)$ and $m \in
(n,\omega)\rangle$ belongs to alg$(\bold x \rest [n_*,\omega)$ hence
by the assumption ``IND$(\bold x \rest [n_*,\omega))$ there is $\bar t
= \langle t_n:n \in [n_*,\omega)\rangle \in \prod\limits_{n \ge n_*}
Y_n$ such that $t_n \notin F_{n,m}(\bar t \rest (n,m))$ when $n_* \le
n < \omega$.  Now by downward induction on $n < n_*$ we choose $t_n
\in Y_n$ such that $t_n \notin F_{n,m}(\langle t_k:k \in
[n+1,m])$ for $m \in [n+1,\omega)$.  This is possible as the countable
union of members of dual$(D_n)$ is not equal to $Y_n$.  We can carry the
induction and $\langle t_n:n < \omega\rangle$ is as required to verify
IND$(\bold x)$.
\end{PROOF}

\begin{PROOF}{\ref{k6}}  
\underline{Proof of Theorem \ref{k6}}

We concentrate on proving part (1), part (2) is easier.

Assume this fails.  So for some $n_* < \omega$ for every $n
\in [n_*,\omega)$ there is a counter-example.  As AC$_{\aleph_0}$
   holds we can find a sequence $\langle \bar{\cF}_n:n \in
   [n_*,\omega)\rangle$ such that:
\mn
\begin{enumerate}
\item[$\odot$]  for $n \in [n_*,\omega)$
\begin{enumerate}
\item[$(a)$]  $\bar{\cF}_n = \langle \cF_{n,\varepsilon}:\varepsilon
\le \zeta\rangle$
\sn
\item[$(b)$]  $\cF_{n,\varepsilon} \subseteq {}^{Y_n} \zeta$ is
non-empty
\sn
\item[$(c)$]  $\bar{\cF}_n$ is a $<_{D_n}$-increasing sequence of sets, i.e. 
$\varepsilon_1 < \varepsilon_2 \le \zeta \wedge f_1 \in
\cF_{n,\varepsilon_1} \wedge f_2 \in \cF_{n,\varepsilon_2} \Rightarrow
f_1 <_{D_n} f_2$.
\end{enumerate}
\end{enumerate}
\mn
Now by AC$_{\aleph_0}$ we can choose $\langle f_n:n \in [n_*,\omega)\rangle$
such that $f_n \in \cF_{n,\zeta}$ for $n \in [n_*,\omega)$.
\mn
\begin{enumerate}
\item[$(*)$]  without loss of generality $n_* = 0$.
\end{enumerate}
\mn
[Why?  As $\bold x \rest [n_*,\omega)$ satisfies the assumptions on
$\bold x$ by \ref{k8}.]

Let
\mn
\begin{enumerate}
\item[$\boxplus_1$]  for $m \le n < \omega$ let $Y^0_{m,n} =
\prod\limits_{\ell=m}^{n-1} Y_\ell$ and for $m,n < \omega$ let
$Y^1_{m,n} := \cup\{Y^0_{k,n}:k \in [m,n]\}$ so $Y^0_{m,n} =
\emptyset = Y^1_{m,n}$ if $m > n$ and $Y^0_{m,n} = \{<>\} = Y^1_{m,n}$
if $m=n$
\sn
\item[$\boxplus_2$]  for $m \le n$ let
$\cG^1_{m,n}$ be the set of functions $g$ such that:
\begin{enumerate}
\item[$(a)$]  $g$ is a function from $Y^1_{m,n}$ into $\zeta + 1$
\sn
\item[$(b)$]  $\langle \rangle \ne \eta \in Y^1_{m,n} \Rightarrow
g(\eta) < \zeta$
\sn
\item[$(c)$]   if $k \in [m,n)$ and 
$\eta \in Y^0_{k+1,n}$ then the sequence $\langle
g(\eta \char 94 \langle y \rangle):y \in Y_k\rangle$ belongs to
$\cF_{k,g(\eta)}$ 
\end{enumerate}
\item[$\boxplus_3$]  $\cG^1_{m,n,\varepsilon} := \{g \in
\cG^1_{m,n}:g(\langle\rangle) = \varepsilon\}$ for $\varepsilon \le
\zeta$ and $m \le n < \omega$.
\end{enumerate}
\mn
Now the sets $\cG^1_{m,n}$ are non-trivial, i.e.
\mn
\begin{enumerate}
\item[$\boxplus_4$]  if $m \le n$ and $\varepsilon \le \zeta$ then
$\cG^1_{m,n,\varepsilon} \ne \emptyset$.
\end{enumerate}
\mn
[Why?  We prove it by induction on $n$; first if $n=m$ this is
trivial.  The unique function $g$ with domain $\{<>\}$ and value
$\varepsilon$.  Next, 
if $m<n$ we choose $f \in \cF_{n-1,\varepsilon}$ hence the
sequence $\langle \cG^1_{m,n-1,f(s)}:s \in Y_{n-1}\rangle$ is well
defined and by the induction hypothesis each set in the sequence is
non-empty.  As AC$_{Y_{n-1}}$ holds there is a sequence $\langle g_s:s
\in Y_{n-1}\rangle$ such that $s \in Y_{n-1} \Rightarrow g_s \in
\cG^1_{m,n-1,f(s)}$.  Now define $g$ as the function with domain $Y^1_{m,n}$:

\[
g(\langle \rangle) = \varepsilon
\]

\[
g(\langle s \rangle \char 94 \nu) = g_s(\nu) \text{ for } \nu \in
Y^1_{m,n-1}.
\]
\mn
It is easy to check that $g \in \cG^1_{m,n,\varepsilon}$ indeed so
$\boxplus_4$ holds.]
\mn
\begin{enumerate}
\item[$\boxplus_5$]  if $g,h \in \cG^1_{m,n}$ and $g(\langle \rangle)
< h(\langle \rangle)$ \then \, there is an $(m,n)$-witness $Z$ for $(h,g)$
which means (just being an $(m,n)$-witness means we omit (d)):
\begin{enumerate}
\item[$(a)$]  $Z \subseteq Y^1_{m,n}$ is closed under initial
segments, i.e. if $\eta \in Y^0_{k,n} \cap Z$ and $m \le k < \ell \le n$
then $\eta \rest [\ell,n) \in Y^0_{\ell,n} \cap Z$
\sn
\item[$(b)$]  $\langle \rangle \in Z$
\sn
\item[$(c)$]  if $\eta \in Z \cap Y^0_{k+1,n},m \le k < n$ then $\{y
\in Y_k:\eta \char 94 \langle y \rangle \in Z\} \in D_k$
\sn
\item[$(d)$]  if $\eta \in Z$ then $g(\eta) < h(\eta)$.
\end{enumerate}
\end{enumerate}
\mn
[Why?  By induction on $n$, similarly to the proof of $\boxplus_4$.]
\mn
\begin{enumerate}
\item[$\boxplus_6$]  $(a) \quad$ we can find $\bar g = \langle g_n:n <
\omega\rangle$ such that $g_n \in \cG^1_{0,n,\zeta}$ for $n < \omega$
\sn
\item[${{}}$]  $(b) \quad$ for $\bar g$ as above
for $n < \omega,s \in Y_n$ let $g_{n+1,s}
\in \cG^1_{m,n}$ be defined by 

\hskip25pt $g_{n+1,s}(\nu) = g_{n+1}(\langle s\rangle \char 94 \nu)$.
\end{enumerate}
\mn
[Why?  Clause (a) by $\boxplus_4$ as AC$_{\aleph_0}$ holds, clause (b)
is obvious by the definitions in $\boxplus_2 + \boxplus_3$.]

We fix $\bar g$ for the rest of the proof
\mn
\begin{enumerate}
\item[$\boxplus_7$]  there is $\left < \langle Z_{n,s}:s \in Y_n \rangle:n
< \omega \right >$ such that $Z_{n,s}$ witness $(g_n,g_{n+1,s})$ 
for $n < \omega,s \in Y_n$.
\end{enumerate}
\mn
[Why?  For a given $n < \omega,s \in Y_n$ we know that
$g_{n+1}(\langle s \rangle) < \zeta = g_n(\langle \rangle)$
hence $Z_{n,s}$ as required exists by $\boxplus_5$.  By AC$_{Y_n}$ for each $n$
a sequence $\langle Z_{n,s}:s \in Y_n \rangle$ as required exists, and
by AC$_{\aleph_0}$ we are done.]
\mn

\begin{enumerate}
\item[$\boxplus_8$]  $Z_n := \{\langle s \rangle \char 94 \nu:s \in
Y_{n-1},\nu \in Z_{n-1,s}\}$ is a $(0,n)$-witness
\sn
\item[$\boxplus_9$]  there is $\bar F$ such that:
\begin{enumerate}
\item[$(a)$]  $\bar F = \langle F_{m,n}:m < n < \omega \rangle$
\sn
\item[$(b)$]   $F_{m,n}:Y^1_{m+1,n} \rightarrow \text{ dual}(D_m)$
\sn
\item[$(c)$]   $F_{m,n}(\nu)$ is $\{s \in Y_m:\nu \char 94 \langle s
\rangle \notin Z_{n-1}\}$ when $\nu \in Z_n$ and is $\emptyset$ otherwise.
\end{enumerate}
\end{enumerate}
\mn
[Why?  Check.]
\mn
\begin{enumerate}
\item[$\boxplus_{10}$]  $\bar F$ witness IND$(\langle (Y_n,D_n):n
<\omega\rangle)$ fail.
\end{enumerate}
\mn
[Why?  Clearly $\bar F = \langle F_{m,n}:m < n < \omega \rangle$ has
the right form.

So toward contradiction assume $\bar t = \langle t_n:n < \omega\rangle
\in \prod\limits_{n < \omega} Y_n$ is such that
\mn
\begin{enumerate}
\item[$(*)_1$]  $m < n < \omega \Rightarrow t_m \notin F_{n,n}(\bar t \rest
(m,n])$.
\end{enumerate}
\mn
Now
\mn
\begin{enumerate}
\item[$(*)_2$]  $\bar t \rest [m,n) \in Z_n$ for $m \le n < \omega$.
\end{enumerate}
\mn
[Why?  By $(*)_1$ recalling $\boxtimes_9(c)$.]
\mn
\begin{enumerate}
\item[$(*)_3$]  $g_{n+1}(\bar t \rest [m,n]) < g_n(\bar t \rest
[m,n))$.
\end{enumerate}
\mn
[Why?  Note that $Z_{n,t_n}$ is a witness for $(g_n,g_{n+1,t_n})$ by
$(*)_2 + \boxplus_7$.  So by $\boxplus_5$ we have $\eta \in
Z_{n,t_n} \Rightarrow g_{n+1,t_n}(\eta) < g_n(\eta)$.  But $m < n
\Rightarrow \bar t \rest [m,n] \in Z_n \Rightarrow \bar t \rest [m,n) \in
Z_{n,t_n}$ hence $g_{n+1}(\bar t \rest [m,n]) = g_{n+1,t_n}(\bar t \rest
[m,n)) < g_n(\bar t \rest [m,n))$ as required.]  

So for each $m < \omega$ the
sequence $\langle g_n(\bar t \rest [m,n):n < \omega\rangle$ is a
decreasing sequence of ordinals, contradiction.  Hence there is no
$\bar t$ as above, so indeed $\boxplus_{10}$ holds. But
$\boxplus_{10}$ contradicts an assumption, so we are done.
\end{PROOF}

\begin{remark}
\label{k10}
1) Note that in \ref{k6} there were no use of completeness demands
except of $\aleph_1$-complete when we get rid of $n^*$, still natural to assume
$\aleph_1$-completeness because: 
if $D'_n$ is the $\aleph_1$-completion of $D_n$ then 
IND$(\langle D'_n:n < \omega\rangle)$ is equivalent to IND$(D_n:n <
\omega)$.

\noindent
2) Recall that by \cite{Sh:513}, iff pp$(\aleph_\omega) >
\aleph_{\omega_1}$ \then \, for every $\lambda > \aleph_\omega$ for
infinitely many $n < \omega$ we have 
$(\forall \mu_1 < \lambda)(\text{cf}(\mu) =
\aleph_n \Rightarrow \text{ pp}(\mu) \le \lambda)$.
\end{remark}

\begin{claim}
\label{k13}
{\rm [DC]}  For $\bold x = \langle Y_n,D_n:n < \omega\rangle$ with each
$D_n$ being an $\aleph_1$-complete filter on $Y_n$, each of 
the following is a sufficient condition for {\rm IND}$(\bold x)$, 
letting $Y(<n) =  \prod\limits_{m<n} \, Y_m$
\mn
\begin{enumerate}
\item[$(a)$]  $\bullet \quad D_n$ is a $(\le (\prod\limits_{m<n}
  D_m)^{Y(<n)})$-complete ultrafilter
\sn
\item[$(b)$]  $\bullet \quad D_n$ is a $(\le(\prod\limits_{m<n}
(D_m)^{Y(<n)})$-complete filter
\sn
\item[${{}}$]  $\bullet \quad$ for each $n$ in the 
following game $\Game_{\bold x,n}$
the non-empty player has a

\hskip25pt  winning strategy.  A play last
$\omega$-moves.  In the $k$-th move the 

\hskip25pt  empty player chooses $A_k \in D_n$ and $\langle X_t:t \in
(\sum\limits_{m<n} D_m)^{Y(<n)} \rangle$ 

\hskip25pt  a partition  of $A_k$ and the non-empty player chooses
$t_k \in A_k$.  

\hskip25pt  In the end
the non-empty player wins the play

\hskip25pt  if $\bigcap\limits_{k<\omega} X_{t_k}$ is non-empty
\sn
\item[$(c)$]  like (a) but in the second part the non-empty player
instead $t_k$ chooses $S_k \subseteq
(\sum\limits_{m<n} D_n)^{Y(<n)}$ satisfying $|S_k| \le_X |S|$
and every $D_{\bold x,n}$ is $(\le S)$-complete, $S$ is infinite
\sn
\item[$(d)$]  if $m < n < \omega$ then $D_m$ is $(\le
\prod\limits_{k=m+1}^{n} Y_k)$-complete\footnote{so the $Y_k$'s are
not well ordered!  If $\alpha < \hrtg(Y_n) \Rightarrow D$ is
$|\alpha|^+$-complete then $\alpha^{Y_m}/D_n \cong \alpha$.  If
$\alpha$ is counterexample $D$ project onto a uniform
$\aleph_1$-complete filter on some $\mu \le \alpha$.}
\end{enumerate}
\end{claim}

\begin{proof}
Straight.
\end{proof}
\bigskip

\subsection {Bounds on the Depth} \
\bigskip

We continue \ref{r29}.  We try to get a bound for singulars of
uncountable cofinality say for the depth, recalling that depth, rank
and ps-$T_D$ are closely related.  

\begin{hypothesis}
\label{c1}
$D$ an $\aleph_1$-complete filter on a set $Y$.
\end{hypothesis}

\begin{remark}
Some results do not need the $\aleph_1$-completeness.
\end{remark}

\begin{claim}
\label{c2}
Assume $\bar\alpha \in {}^Y\text{\rm Ord}$.

\noindent
1) {\rm [DC]} (No-hole-Depth)  If $\zeta + 1 \le 
\text{\rm ps-o-Depth}^+_D(\bar\alpha)$ \then \, for some $\bar\beta  \in
   {}^Y\text{\rm Ord}$, we have $\bar\beta \le \bar\alpha$ {\rm mod} $D$ and
$\zeta +1 = \text{\rm ps-o-Depth}^+(\bar\beta)$.

\noindent
2) In Definition \ref{r26} we may allow $\cF_\varepsilon \subseteq
   {}^Y\Ord$ such that $g \in \cF_\varepsilon \Rightarrow g < f \mod D$.

\noindent
3) If $\bar\beta \in {}^Y\text{\rm Ord}$ and $\bar\alpha = \bar\beta$
   {\rm mod} $D$ \then \, {\rm ps-o-Depth}$^+(\bar\alpha) = \text{\rm
   ps-o-Depth}^+(\bar\beta)$.

\noindent
4) If $\{y \in Y:\alpha_y = 0\} \in D^+$ \then \, {\rm
   ps-o-Depth}$^+(\bar\alpha)=1$.

\noindent
5) Similarly for the other  versions of depth from Definition
   \ref{k1}.
\end{claim}

\begin{PROOF}{\ref{c2}}
1) By DC \wilog \, there is no $\bar\beta <_D \bar\alpha$ such that
   $\zeta +1 \le \text{ ps-o-Depth}^+(\bar\beta)$.  Without loss of
   generality $\bar\alpha$ itself fails the desired conclusion hence
   $\zeta +1 < \text{ ps-o-Depth}^+(\beta)$.  By parts (3),(4) \wilog
   \, $y \in Y \Rightarrow \alpha_y > 0$.  As $\zeta + 1 < \text{
   ps-o-Depth}^+(\bar\alpha)$ there is a $<_D$-increasing sequence $\langle
   \cF_\varepsilon:\varepsilon < \zeta +1 \rangle$ with
   $\cF_\varepsilon$ a non-empty subset of $\Pi\bar\alpha$.  Now any
   $\bar\beta \in \cF_\zeta$ is as required: $\zeta +1 \le \text{
   ps-o-Depth}^+(\bar\beta)$ as witnessed by $\langle
   \cF_\varepsilon:\varepsilon < \zeta \rangle$, recalling part (2);
   contradicting the extra assumption on $\bar\alpha$ (being $<_D$-minimal
   such that...).

\noindent
2) Let $\cF'_\varepsilon = \{f^{[\bar\alpha]}:f \in \cF_\varepsilon\}$
   where $f^{[\bar\alpha]}(y)$ is $f(y)$ if $f(y) < \alpha_y$ and is
   zero otherwise.

\noindent
3),4) Obvious.

\noindent
5) Similarly.
\end{PROOF}

\begin{claim}
\label{c4}
[{\rm DC + AC}$_Y$]  If 
$\bar\alpha,\bar\beta \in {}^Y\text{\rm Ord}$ and $D$ is a
filter on $Y$ and $y \in Y \Rightarrow |\alpha_y| = |\beta_y|$
\then \, {\rm ps}-$\bold T_D(\bar\alpha) = \text{\rm ps}-\bold
T_D(\bar\beta)$.   
\end{claim}

\begin{proof}
Straight.
\end{proof}

Assuming full choice the following is a version of Galvin-Hajnal theorem.
\begin{theorem}
\label{c7}
[{\rm DC + AC}$_Y$]
Assume $\alpha(1) < \alpha(2) < \lambda^+$, {\rm
ps-o-Depth}$^+(\lambda) \le \lambda^{+\alpha(1)}$ and $\xi =
\hrtg({}^Y \alpha(2)/D)$.  \Then \, {\rm ps-o-Depth}$^+_D(\lambda^{+
\alpha(2)}) < \lambda^{+(\alpha(1)+\xi)}$.
\end{theorem}

\begin{PROOF}{\ref{c7}}
Let $\Lambda = \{\mu:\lambda^{+\alpha(1)} < \mu \le 
\lambda^{\alpha(1)+ \xi}\}$ and for every $\mu \in \Lambda$ let
\mn
\begin{enumerate}
\item[$(*)_1$]  $\cF_\mu = \cF(\mu) = 
\{f:f \in {}^Y\{\lambda^{+ \alpha}:\alpha <
\alpha(2)\}$ and $\mu = \text{ ps-Depth}^+_D(f)\}$
\sn
\item[$(*)_2$]  obviously $\langle \cF_\mu:\mu \in \Lambda\rangle$ is
a sequence of pairwise disjoint subsets of ${}^Y\alpha(2)$ closed
under equality modulo $D$.
\end{enumerate}
\mn
By the no-hole-depth claim \ref{c2}(1) above we have
\mn
\begin{enumerate}
\item[$(*)_3$]  if $\mu_1 < \mu_2$ are from $\Lambda$ and $f_{\mu_2} \in
\cF_2$ \then \, for some $f_1 \in \cF_{\mu_1}$ we have $f_1 < f_2$ mod $D$
\sn
\item[$(*)_4$]  $\xi > \sup\{\zeta+1:\cF(\lambda^{+(\alpha+\zeta)}) \ne
\emptyset\}$ implies the conclusion.
\end{enumerate}
\mn
Lastly, as $\xi = \hrtg({}^{Y_{\alpha(2)}}/D)$ we are done.
\end{PROOF}

\begin{remark}
0) Compare this with conclusion \ref{r21}.

\noindent
1) We may like to lower $\xi$ to ps-Depth$^+_D(\alpha(2))$, toward
   this let
\mn
\begin{enumerate}
\item[$(*)_1$]   $\cF'_\mu = \{f \in \cF_\mu:$ there is no $g \in
\cF_\mu$ such that $g < f$ mod $D\}$.
\end{enumerate}
\mn
By DC
\mn
\begin{enumerate}
\item[$(*)_2$]   if $f \in \cF_\mu$ then there is $g \in \cF'_\mu$ 
such that $g \le_D f$ mod $D$.
\end{enumerate}
\mn
2) Still the sequence of those $\cF'_\mu$ are not $<_D$-increasing.
\end{remark}

Instead of counting cardinals we can count regular cardinals.
\begin{theorem}
\label{c13}
[{\rm DC+AC}$_Y$]
The number of regular cardinals in the interval

\noindent
$(\lambda^{+\alpha(1)},\text{\rm
ps-depth}^+_{Y,\kappa)}(\lambda^{+\alpha(2)})$ is at most 
$\hrtg({}^Y\alpha(2)/D)$ \when \,:
\mn
\begin{enumerate}
\item[$(a)$]  $\alpha(1) < \alpha(2) < \lambda^+$
\sn
\item[$(b)$]  $\kappa > \aleph_0$
\sn
\item[$(c)$]  $D$ is a $\kappa$-complete filter on $Y$
\sn
\item[$(d)$]  $\lambda^{+\alpha(1)} = 
\text{\rm ps-Depth}_{(Y,\kappa)}(\lambda)$.
\end{enumerate}
\end{theorem}

\begin{PROOF}{\ref{c13}}
Straight, using the No-Hole Claim \ref{r22}.
\end{PROOF}
\bigskip

\subsection {Concluding Remarks} \
\bigskip

Those are comments to \cite{Sh:938}.

\begin{definition}
We ay $(\Pi \bar\alpha,<_{D_*})$ has weak $\kappa$-true cofinality
$\delta$, omitting $\kappa$ means $\kappa = \aleph_0$, if there is
some witness or $(\bbD,\bar f)$ which means:
\mn
\begin{enumerate}
\item[$(a)$]  $\bbD \subseteq \{D:D$ an $\kappa$-complete filter on
  $Y$ extending $D\}$
\sn
\item[$(b)$]  $D_* = \cap\{D:D \in \bbD\}$
\sn
\item[$(c)$]  $\bar{\cF} = \langle \cF_{D,\alpha}:D \in bbD,\alpha <
  \partial\rangle$
\sn
\item[$(d)$]  $\langle \cF_{D,\alpha}:\alpha < \delta\rangle$ witness
  $(\Pi \bar\alpha,<_D)$ has pseudo-true-cofinality $\delta$.
\end{enumerate}
\end{definition}

\begin{definition}
$\delta = \wtcf_\kappa(\Pi \bar\alpha,<_{D_*})$ means $(\Pi
  \bar\alpha,<_{D_*})$ has weak $\kappa$-true cofinality $\delta$ and
  $\delta$ is minimal (hence a regular cardinal).
\end{definition}

\begin{discussion}
1) Why do not ask $\delta$ to be regular always?  We may consider a
   sequence of $\delta$'s and as in $\id - \cf_\kappa(\bar\alpha)$ in
   \cite{Sh:F1078}.

\noindent
2) Can we $(\ZF + \DC + \AC_{\beth_\omega})$ prove \cite{Sh:460},
   using $p T_D(\bar\alpha)$?  Use \cite[\S1]{Sh:460}.

\noindent
3) Can we generalize the proof of \cite[\S1]{Sh:824} using $\ps-\bold
   T_D(f)$?  We get $\lambda$ is $\ps-\bold T_D(\prod\limits_{i <
   \kappa} \lambda_i),\kappa< \mu$ as witnessed by $\langle
   \cF^+_\alpha:\alpha < \lambda\rangle$, but toward contradiction we
   have $D_n \in \Fil^1_{\kappa^+_n}(\kappa_{n+1})$.
\end{discussion}

\begin{remark}
For $D \in \Fil^1_\kappa(Y),\ps - \bold T_D(f)$ is closely related to
$\sup\{\ps - \bold T_{D_1}(f)$.  $D_1$ is a filter on some $\theta <
\hrtg(Y)$ such that $D_2 \le_{\RK} D$ so natural to define $\ps -
\bold T_{\bbD}$.
\end{remark}

\begin{definition}
1) Assume $\bbD_1$ is a set of filters and let $\pry(\bbD_1)$ be

\begin{equation*}
\begin{array}{clcr}
\{D_2: & \text{ for some } D_1 \in \bbD_2,\mu < \hrtg(\Dom(D_2))
\text{ and} \\
  &h:\Dom(D_1) \rightarrow \mu \text{ we have } D_2 = h(D_1)\}.
\end{array}
\end{equation*}

\mn
2) Let $\ps - \bold T_{\bbD}(\bar\alpha) = \sup\{\ps - \bold
T_D(\bar\alpha):D \in \bbD\}$.

\noindent
3) FILL.
\end{definition}

\begin{claim}
Let $\bbD_1$ be a set of $\aleph_1$-complete filters, $\bbD_2 =
\pry(\bbD_1)$.  Then the following cardinals are $S$-almost equivalent
where $S = \Fil^1_{\aleph_1}(\bbD_1) =
\cup\{\Fil^1_{\aleph_1}(D_1):D_1 \in \bbD_1\}$:
\mn
\begin{enumerate}
\item[$(a)$]  $\ps - \bold T_{\bbD_1}(\bar\alpha)$
\sn
\item[$(b)$]  $\ps - \bold T_{\bbD_1}(\bar\alpha)$
\sn
\item[$(c)$]  FILL.
\end{enumerate}
\end{claim}
\newpage

\section {On RGCH with little choice}

If we assume $\Ax_4$ the answer seems yes. We like to say that just
assuming DC; so if we have enough cases of IND, we use
\cite[\S(3B)]{Sh:955} if not, assume for every $\kappa$ we have $\bold
p$ more or less as in \cite[3.1]{Sh:938}, i.e. omitting the ranks such
that $(\forall \lambda)\Ax^\theta_{\lambda,\mu_0,\aleph_0}$ all $D \in
\bbD_{\bold p}$ are $\mu_1$-complete.  We try to repeat.

So trying to immitate, e.g. \cite{Sh:829} in the main case we have
$\bold d \in \bbD_{\bold p},\bar\alpha \in {}^{Y(\bold d)}\alpha$.
\Wilog \, $(\forall t \in Y_{\bold d})(\alpha_t,\bold i_1)$ is as
required, using the induction hypothesis.

For $s \in Y_{\bold d}$, using $c \ell:[\alpha]^{< \mu(\bold P)}
\rightarrow [\alpha]^{< \mu(\bold P)}$ which exists by
$\Ax^0_\lambda,\ldots$ we have $\big < \langle f^t_{\bold e,\bold
  y,\beta}:\beta < \alpha_t\rangle:\bold e \in \bbD_{\ge \bold
  i_1},\beta,\bold y \in \Fil^4_{\kappa(\bold i_1,\bold
  p)}(D_{e_x})\big >$ such that every: if $\bold d \in \bbD_{\ge \bold
  i_1},s \in Y_{\bold d},f \in {}^{Y[\bold e]}(\alpha_s)$ then for
some set $\langle (y_i,\beta_i):i < \zeta_p < \kappa(\bold i,\bold
p)\rangle,\bigwedge\limits_{t \in Y_{\bold e}} \, \bigvee\limits_{i}
f(s) = f_{\bold e,\bold y_i,\beta_i}(s)$.

Why?  Given $(\bold e,f)$ if there is no such sequence, we can find a
filter $\kappa(\bold i_1,\bold p)$-complete filter on $Y_{\bold e}$
such that...

But we need more: given $\bar f = \langle f_s:s \in Y_{\bold
  s}\rangle,f_s \in {}^{Y[\bold e]}\alpha_s$ and we like to consider
  all $f_s$ simultaneously, say find $\langle (y_{s,i},\beta_{s,i}):s
  \in Y_s,i < i_s\rangle$ as above.

If we have $\bold d \in \bbD_{\bold p} \Rightarrow \AC_{Y_{\bold d}}$
this can be done.  So the status of $\Ax^0_\lambda$ change: given
$\bold p$ we say?  If $(\forall x A^0_y,\ldots)$ fix.  If not, then
for some $\lambda(*)$ we have $\bold i < \cf(\mu) \Rightarrow \neg
\Ax^0_{\lambda,\kappa(\bold i,\bold p),\aleph_1}$ (can determine the
other cases).

We get
\mn
\begin{enumerate}
\item[$(*)_1$]  if $\partial < \mu_{\bold p}$ then $I = [\lambda]^{<
  \partial}$ and $D_n = \dual(I)$ then $\IND(\langle I,D:n\rangle)$.
\end{enumerate}

\begin{question}
Can we use $\langle ([\lambda]^{\kappa(i,\bold
  p)},I_{\lambda,\kappa(\bold i,p)}):n < \omega\rangle$?

Can we avoid using $\langle \AC_{\kappa(i,\bold p)}:i <
\cf(\mu)\rangle$?  Given $\bar f = \langle f_\alpha:s \in Y_{\bold
  d}\rangle$ we can consider $Y_* = Y_{\bold d} \times Y_{\bold e}$
and for every sequence $\bold x = \langle (\bold y_s,f_s):s \in
Y_{\bold d}\rangle,f_s \in {}^{Y[\bold e]}(\alpha_s)$ let $A_{\bold x}
= \{(s,t)(Y_{\bold d} \times Y_{\bold e}):f_s(t) = f_s(t)\}$.

Now we may look at ($R$ not too large)

\[
D^* = \{Z \subseteq Y_*: \text{ there is } \langle \bold x_r:r \in P\rangle
\text{ such that } Y \backslash Z \subseteq 
\bigcup\limits_{r \in R} A_{\bold x_r}\}.
\]

\mn
So $D_R$ is a $\kappa_p(i_1,\bold p)$-complete filter.

Let $D^*_{R,s}$ be the projection of $D^*_R$ to $\{s\} \in 
Y_{\bold e}$.  Clearly it is the filter defined by $(\alpha_s,f_s)$.
\end{question}

\noindent
Recall \cite[2.2]{Sh:835}.
\begin{definition}
\label{n2}
We say $\Ax^0_{\alpha,\delta}$ when some $c \ell$ exemplifies it which
means:
\mn
\begin{enumerate}
\item[$(a)$]  $c \ell:[\alpha]^{< \kappa} \rightarrow [\alpha]^{<
  \mu}$
\sn
\item[$(b)$]  $u \subseteq c \ell(u)$
\sn
\item[$(c)$]  $u_1 \subseteq u_2 \Rightarrow c \ell(u_1) \subseteq c
  \ell(u_2)$
\sn
\item[$(d)$]  there is no sequence $\langle \alpha_n:n < \omega\rangle
  \in {}^\omega \alpha$ such that $\alpha_n \notin c \ell\{\alpha_k:k
  > n\}$.
\end{enumerate}
\end{definition}

\begin{definition}
\label{n4}
We say $\bold x$ is a filter system (like \cite[3.1]{Sh:938}, add
$\kappa_{\bold p,\bold d},\Rep_{\kappa(\bold d,\bold p)}(D_{\bold p})$
but no $\rk$
\mn
\begin{enumerate}
\item[$(a)$]    $\mu$ is singular
\sn
\item[$(b)$]   each $\bold d \in \bbD$ is (or just we can compute
from it) a pair $(Y,D) = (Y_{\bold d},D_{\bold d}) = (Y[\bold
d],D_{\bold d}) = (Y_{\bold p,\bold d},D_{\bold p,\bold d})$ such that:
\sn
\begin{enumerate}
\item[$(\alpha)$]   $\theta(Y_{\bold d}) < \mu$, on $\theta(-)$
see Definition \ref{z0.15}
\smallskip

\noindent
\item[$(\beta)$]   $D_{\bold d}$ is a filter on $Y_{\bold d}$
\end{enumerate}
\sn
\item[$(c)$]   $(\alpha) \quad \kappa_\iota = \kappa_{\bold p,\iota} =
  \kappa(\bold i,\bold p)$ is a cardinal $< \mu$
\sn
\item[${{}}$]  $(\beta) \quad i_1 < i_2 \Rightarrow \kappa_{\bold
  p,i_1} < \kappa_{\bold p,i_2}$
\sn
\item[${{}}$]  $(\gamma) \quad (\forall \sigma < \mu)[\exists i <
  \cf(\mu))(\sigma < kappa_{\bold p,i}$
\sn
\item[${{}}$]  $(\delta) \quad$ if $\bold d \in \bbD_{\ge i}$ then
  $D_{\bold d}$ is $\kappa_{\bold p,i}$-complete
\sn
\item[${{}}$]  $(\varepsilon) \quad \mu$ is strong limit
\sn
\item[$(d)$]   $(\alpha) \quad \Sigma$ is a function with
domain $\bbD$ such that $\Sigma({\bold d}) \subseteq \bbD$
\sn
\item[${{}}$]   $(\beta) \quad$ if ${\bold d} \in \bbD$
and $\bold e \in \Sigma(\bold d)$ then $Y_{\bold e} = 
Y_{\bold d}$ [natural to add $D_{\bold d} \subseteq D_{\bold e}$, 

\hskip25pt this is not demanded but see \ref{m4.9}(2)]
\sn
\item[$(e)$]   $(\alpha) \quad \bold j$ is a function from $\bbD$
onto cf$(\mu)$
\sn
\item[${{}}$]  $(\beta) \quad$ let $\bbD_{\ge i} =
\{\bold d \in \bbD:\bold j(\bold d) \ge i\}$ and $\bbD_i = 
\bbD_{\ge i} \backslash \bbD_{i+1}$
\sn
\item[${{}}$]   $(\gamma) \quad \bold e \in \Sigma(\bold d)
\Rightarrow \bold j(\bold e) \ge \bold j(\bold d)$
\sn
\item[$(f)$]   for every $\sigma < \mu$ for some 
$i < \text{ cf}(\mu)$, if $\bold d \in \bbD_{\ge i}$, \und{then}
$\bold d$ is $(\bold p,\le \sigma)$-complete where
\sn
\item[$(g)$]  $\bold p$ is complete when $\bbD_{\ge i} =
  \{(\kappa,D):\kappa \in [\kappa_{\bold p,i},\mu),D$ a $\kappa_{\bold
  p,i}$-complete filter on $\kappa\}$.
\end{enumerate}
\end{definition}

\begin{definition}
\label{n8}
Let $\Ax^0_{\alpha,\bold p}$ means that: there is $c \ ell$ satisfying
(a)-(c) of \ref{m2} as:
\mn
\begin{enumerate}
\item[$(d)$]  if $\bold d \in \bbD$ and $u \in [\alpha]^{<
  \hrtg(Y[\bold d])}$ \then \, $|c \ell(u)| < \kappa (\bold d,\bold
  p)$.
\end{enumerate}
\mn
FILL
\end{definition}

\begin{claim}
\label{n10}
Assume $\Ax^0_{\alpha,\mu,\le \kappa},D$ a filter on $Y$ and
$\bar\alpha \in {}^Y(\alpha_* +1)$.

Then $\ps-0-\Depth_D(\bar\alpha) \le_S 0-\Depth_D(\bar\alpha)$.

Why?
\end{claim}

\begin{proof}
Let $c \ell$ witness $Ax^0_{\alpha_*}$, and assume $u \in [\alpha]^{<
  \hrtg(Y)} \Rightarrow c \ell(u)$.  Let $\kappa = \sup\{|c
  \ell(u)|^+:u \in [\alpha_*]^{< \hrtg(Y)}\}$.  For transparency as $0
  \notin \Rang(\bar\alpha)$, assume $\beta_* <
  \ps-0-\Depth_D(\bar\alpha)$, so there is a sequence $\langle
  \cF_\beta:\beta < \beta_*\rangle$ witnessing it so $f \in\bold
  F_\beta \Rightarrow f < \bar\alpha$.

For each $\beta < \beta_*$ so $\cF_\beta \subseteq \Pi \bar\alpha,f
\in \cF := \cup\{\cF_\beta:\beta < \beta_*\}$, there is $y \in
\Rep_\kappa(D)$ which represents $f$ which means:
\mn
\begin{enumerate}
\item[$(*)_{f,\bold y}$]  $(a) \quad \bold y \equiv (Y,D,A,h)$
\sn
\item[${{}}$]  $(b) \quad$ if $B \in D$ and $B \subseteq A$ then $c
  \ell\{f(t):t \in \beta\} = c \ell\{f(t):t \in A\}$
\sn
\item[${{}}$]  $(c) \quad h$ is a function with domain $A_{\bold y}$
  such that: $h(t) = \otp(f(t) \cap c \ell\{f(s):s \in A_{\bold
  y}\})$ so $< \mu$
\sn
\begin{enumerate}
\item[$\boxplus$]  if $f_1,f_2 \in \cF_\beta$ are represented by
  $\bold y$ then $f_1 \rest A_{\bold y} = f_2 \rest A_{\bold y}$.
\end{enumerate}
\end{enumerate}
Now
\mn
\begin{enumerate}
\item[$\boxplus$]  $|\Rep_\kappa(D)| = |D \times {}^Y \kappa|$
\sn
\item[$\boxplus$]  for $\bold y \in \Rep_\kappa(D)$ let $\cU_{\bold y}
  = \{\beta < \beta_*$: there is $f \in \cF_\beta$ represented by
  $\bold y\}$
\sn
\item[$\boxplus$]  $\langle \cU_{\bold y}:y \in \Rep(D,\kappa)$ is
  well defined
\sn
\item[$\boxplus$]   $\beta_* = \cup\{\cU_{\bold y}:\bold y \in
 \Rep(D,\kappa)$
\sn
\item[$\boxplus$]  for $\bold y \in \Rep_\kappa(D)$ and $\alpha \in 
\cU_{\bold y}$ let $g_{\bold y,\beta}$ is the unique member of $\Pi
\bar\alpha$ such that: if $f \in \cF_\beta$ is represented by $\bold
y$ then $g_{\bold y,\beta} \rest A_{\bold y} = f \rest A_{\bold y}$
and $g_{\bold y,\beta}(t) = 0$ for $t \in Y \backslash A_{\bold y}$
\sn
\item[$\boxplus$]   $\langle g_{\bold y,\beta}:\beta \in \cU_{\bold
  y}\rangle$ is $<_D$-increasing sequence in $\Pi\bar\alpha$.
\end{enumerate}
\end{proof}

\begin{claim}
\label{n13}
Assume $D_*$ is a $\kappa$-complete filter on $Y,\kappa \ge \aleph_1$
and $\Ax^0_{\lambda,\le \gamma,\le \kappa}$ so $\gamma$ acts as an
ordinal and $\mu = \chi$ and $S = \Fil^4_\kappa(D_*,\gamma)$, so
$\partial$ fixes the order type of $c \ell(\{f(s):s \in Y\})$ and
$\bbD = \{\dual(J[f,D]):f \in {}^Y \Ord\}$.

The following cardinals are $S$-almost equal for $\bar\alpha \in {}^Y
\Ord$
\mn
\begin{enumerate}
\item[$(a)$]  $0-\Depth^+_{\bbD}(\bar\alpha)$
\sn
\item[$(b)$]  $\ps-0-\Depth(\bar\alpha)$
\sn
\item[$(c)$]  $\ps-\bold T_{\bbD}(\bar\alpha)$
\sn
\item[$(d)$]  $\cup \sup\{\rk_D(\bar\alpha)+1:D \in \bbD\}$.
\end{enumerate}
\end{claim}

\begin{proof}
FILL.
\end{proof}

\begin{theorem}
\label{m8}
Let $\bold p = (\bbD,\mu,\ldots)$ be a filter system and $(\forall
alpha)(\forall^\infty,< \cf(\mu))(\Ax^0_{\lambda,\kappa_{\bold
    p,i+1},\aleph_1})$.

Assume further $\AC_{\kappa(i,\bold p)}$ for $i < \cf(\mu)$.  For
$\bold d \in \bbD_{\bold p}$ let obey??

For every $\alpha$ (question: or $\lambda$?) such that
$\Ax^0_{\alpha,\kappa(\bold p),\aleph_1}$ there is $i < \cf(\mu_{\bold
  p})$ such that: if $\bold d \in \bbD_{\ge i}$ then the following as
$\Rep_{\kappa(\bar{\bold d},\bold p)}(D_{\bold d})$-almost equal
\mn
\begin{enumerate}
\item[$(a)$]  $\alpha$
\sn
\item[$(b)$]  $0-\Depth_{D_{\bold d}}(\alpha)$
\sn
\item[$(c)$]  $\ps-0-\Depth_{D_{\bold d}}(\bar\alpha)$
\sn
\item[$(d)$]  $\ps-\bold T_{D_{\bold d}}(\alpha)$
\sn
\item[$(e)$]  $\rk_{D_{\bold d}}(\alpha)$.
\end{enumerate}
\end{theorem}

\begin{remark}
%\label{}
1) For (b),(c) their being almost equal we already know, see xxx.

\noindent
2) Use $\rk_{\bold d}$ or $\rk_{D_{\bold d}}$?  Presently, $\rk_d$.
\end{remark}

\begin{proof}

\noindent
\underline{Case 1}:  $\alpha < \mu$

Obvious.
\medskip

\noindent
\underline{Case 2}: $\alpha < \mu^+$

Easy.
\medskip

\noindent
\underline{Case 3}:  $\alpha \ge \mu^+$ and for $\bold d \in \bbD$ and
$\bar\alpha \in {}^{Y[\bold d]}\alpha$ do we have $\alpha <
\ps-0-\Depth(\bar\alpha)$.

Easy by the definitions.
\medskip

\noindent
\underline{Case 4}:  as ab there are $\bold d \in \bbD$ and
$\bar\alpha \in {}^{Y[\bold d]}\alpha$.

Choose $\langle g^*_\varepsilon:\varepsilon < \alpha\rangle$ witness
$\alpha < 0-\Depth^+_{D_{\bold d}}(\alpha)$ or more: such that
$J[g^*_\varepsilon,D_{\bold d}]$ is constant; $D_2$ the dual.

For $s \in Y_{\bold d}$ clearly $\bold i(s) = \min\{i < \cf(\mu)$: for
$\alpha_t,i$ is as required in the claim$\}$.  Clearly $\bold i(s) <
\cf(\mu)$ is well defined by the induction hypothesis
\mn
\begin{enumerate}
\item[$(*)$]  \wilog \, for some $\bold i_0,A = \{s \in Y_{\bold
  d}:\bold i(s) = \bold i-)\} \in D_{\bold d}$.
\end{enumerate}
\mn
[Why?  See Definition \ref{m4}, clause $(*)$?]

We choose $\bold i_1 \in (\bold i_0,\cf(\mu)))$ such that
\mn
\begin{enumerate}
\item[$(*)$]  FILL.
\end{enumerate}
\mn
Now let $\bold e \in \bbD_{\ge \bold i_1}$ and $\beta_* <
0-\Depth^+(\alpha)$ and let $\langle f_\beta:\beta < \beta_*\rangle$
witness this.  Define $\langle f_{\beta,s}:\beta < \beta_*,s \in
Y_{\bold d}\rangle$ with $f_{\beta,s}$ the function from $Y_{\bold e}$
into $\alpha_*$ defined by $f_{\beta,s}(t) = g_{f_\beta(t)}(s)$.

let $\langle \xi_{\beta,s}:\beta < \beta_*,s \in Y_{\bold d}\rangle$
be defined by
\mn
\begin{enumerate}
\item[$\bullet$]  $\xi_{\beta,s} = \rk_{D_{\bold e}}(f_{\beta,s})$.
\end{enumerate}
\mn
Now
\mn
\begin{enumerate}
\item[$(*)$]  $\xi_{\beta,s} < \alpha_s$.
\end{enumerate}
\mn
Lastly, let $\langle \xi_\beta:\beta < \beta_*\rangle$ be defined by
\mn
\begin{enumerate}
\item[$\bullet$]  $\xi_\beta = \rk_{\bold d}(\bar\xi_\beta)$ where
  $\bar\xi_\beta = (\xi_{\beta,s}:s \in Y_{\bold d})$.
\end{enumerate}
\mn
As $\rk_{\bold d}(\bar\alpha = \alpha$ and $\langle \xi_{\beta,s}:s
\in Y_{\bold d}\rangle <_{D_{\bold d}} \bar \alpha$ we have
\mn
\begin{enumerate}
\item[$(*)$]  $\xi_\beta \le \alpha$ (or $\xi_\beta < \alpha$).
\end{enumerate}
\mn
Now for each $\xi \le \alpha$ let
\mn
\begin{enumerate}
\item[$(*)$]  $u_\xi = \{\beta < \beta_*:\xi_\beta = \xi\}$.
\end{enumerate}
\mn
It suffices (check formulation) to prove
\mn
\begin{enumerate}
\item[$\boxplus$]  $|u_\xi| < \hrtg(\Fil^1_{\aleph_1}(D_{\bold d})
  \times \Fil^1_{\aleph_1}(D_{\bold e}))$.
\end{enumerate}
\mn
Why?  For every $\beta < \beta_*$ let $\bold x^1_\beta = (J\langle
\bar\xi_{\beta,s}:s \in Y_{\bold d}\rangle,D_{\bold d}),\bold
x^2_\beta = \langle J[\langle g^*_{f_{\beta,s}(t)}(s):
t \in D_{\bold e}\rangle,D_{\bold e}]:
s \in Y_{\bold d}\rangle,\bold x^3_\beta =
J[f_\beta,D_{\bold e},\bold x^4_\beta = \langle
  J[g^*_{f_{\beta,t}},D_{\bold d}]:t \in Y_{\bold e}\rangle$.]

Now
\mn
\begin{enumerate}
\item[$\bullet$]  if $\beta_1 < \beta_1 < \beta_*$ and
  $(\xi_{\beta_1},\bold x^\ell_{\beta_1}) = (\xi_{\beta_2},\bold
  x^\ell_{\beta_2})$ then $\xi_1 = \xi$.
\end{enumerate}
\mn
[The delicate point: how much should $\bold i_1$ or $\comp(\bold e)$
  be above $\bold d$?  or too similar to \cite[\S2]{Sh:938}.]
\bigskip

\centerline {$* \qquad * \qquad *$}
\bigskip

Let $J = J[\langle \xi_{\beta_\ell,1}:s \in \bold d\rangle,
D_{\bold d}],J_s = J[\langle g_{f_{\beta_\ell,2}(t)}(s):t \in 
D_{\bold d}\rangle]$.

First, note that as $\xi_{\beta_1} = \xi_{\beta_2}$, clearly $A = \{s
\in Y_{\bold d}:\xi_{\beta_{1,s}} = \xi_{\beta_2,s}\} = Y_{\bold d}
\mod J$.  Also for every $s \in A$ we have $B_s := \{t \in Y_{\bold
  e}:g_{f_{\beta_1,s}(t)}() := g_{f_{\beta,s}(t)}(s)\} = Y_{\bold e}
\mod J$.

Is $\bold i_1$ large enough?
\bigskip

\centerline {$* \qquad * \qquad *$}
\mn
\begin{enumerate}
\item[$\bullet$]  $A_{\beta_1,\beta_2} = \{t \in Y_{\bold e}:
f_{\beta_1}(t) < f_{\beta_2}(t)\} = Y_{\bold e} \mod D_{\bold e}$
\sn
\item[$\bullet$]  for $t \in Y_{\bold d}:A^t_{\beta_1,\beta_2} = 
\{s \in Y_{\bold d}:g_{f_{\beta_1}(t)}(s) < g_{f_{\beta_2}(t)}(s)\}$.
\end{enumerate}
\mn
So
\mn
\begin{enumerate}
\item[$\bullet$]  $A_{\beta_1,\beta_2} =  Y_{\bold e} \mod D_{\bold e}$
\sn
\item[$\bullet$]  $A^t_{\beta_1,\beta_2} = Y_{\bold e} \mod D_{\bold d}$
for $t \in A_{\beta_1,\beta_2}$.
\end{enumerate}
\mn
As $\hrtg(D_{\bold d}) < \comp(D_{\bold e})$ by the choice of $\bold
i_2$ and ``$\bold e \in \bbD_{\ge \bold i_1}$", for some $A_* \in
D_{\bold d}$ we have
\mn
\begin{enumerate}
\item[$\bullet$]  $B_* = \{t \in Y_{\bold e}:
A^t_{\beta_1,\beta_2} = A_*\} \ne \emptyset \mod J$ where $J = 
J[f_{\beta_1},D_{\bold e}] = J[f_{\beta_2},D_{\bold e}]$.
\end{enumerate}
\mn
Hence
\mn
\begin{enumerate}
\item[$\bullet$]   for every $s \in A_*,t \in B_*$ we have
$g_{f_{\beta_1}(t)}(s) < g_{f_{\beta_2}(t)} (s)\}$.
\end{enumerate}
\end{proof}
\bigskip

%\bibliographystyle{alphacolon}
%\bibliography{lista,listb,listx,listf,liste,listz}

\end{document}